\documentclass[11pt]{article}

\usepackage[T1]{fontenc}
\usepackage[utf8]{inputenc}
\usepackage{lmodern}
\usepackage{microtype}
\usepackage[a4paper,margin=28mm]{geometry}
\usepackage{amsmath,amssymb,amsthm,mathtools,mathrsfs}
\usepackage{bm}
\usepackage{booktabs}
\usepackage{graphicx}
\usepackage{xcolor}
\usepackage{subcaption}
\usepackage{enumitem}
\usepackage{placeins}
\usepackage[numbers,sort&compress]{natbib}
\usepackage[hidelinks]{hyperref}
\usepackage[capitalise,noabbrev,nameinlink]{cleveref}

\hypersetup{
  pdftitle={High-order energy diminishing ALE-SAV finite element method for two-phase Navier-Stokes flow},
  pdfauthor={Author names omitted for review},
  pdfkeywords={two-phase Navier--Stokes equations, arbitrary Lagrangian--Eulerian finite elements, scalar auxiliary variable, Crank--Nicolson method, surface tension}
}

\graphicspath{{figures/}}
\allowdisplaybreaks
\numberwithin{equation}{section}
\newtheorem{theorem}{Theorem}[section]
\newtheorem{lemma}[theorem]{Lemma}

\theoremstyle{definition}

\newtheorem{remark}[theorem]{Remark}
\crefname{assumption}{assumption}{assumptions}
\Crefname{assumption}{Assumption}{Assumptions}

\newcommand{\R}{\mathbb R}
\newcommand{\Om}{\Omega}
\newcommand{\Gam}{\Gamma}
\newcommand{\uh}{\bm u_h}
\newcommand{\vh}{\bm v_h}
\newcommand{\wh}{\bm w_h}
\newcommand{\Xh}{\bm X_h}
\newcommand{\etah}{\bm\eta_h}
\newcommand{\betah}{\bm\beta_h}
\newcommand{\M}{\mathcal M}

\newcommand{\F}{\mathcal F}
\newcommand{\C}{\mathcal C}
\newcommand{\Rop}{\mathcal R}

\newcommand{\norm}[1]{\left\lVert #1\right\rVert}
\newcommand{\abs}[1]{\left\lvert #1\right\rvert}
\newcommand{\dd}{\,\mathrm d}
\newcommand{\Div}{\operatorname{div}}
\newcommand{\id}{\operatorname{id}}

\def\x{{\bf x}}

\def\K{{\mathcal K}}

\title{
High-order energy diminishing ALE-SAV
finite element methods for two-phase Navier--Stokes flow
}
\author{
Shu Ma\thanks{Department of Mathematics, Hong Kong Baptist University, Kowloon Tong, Hong Kong.}}
\date{}

\begin{document}
\maketitle

\begin{abstract}
We develop a high-order fully discrete ALE-SAV finite element method for sharp-interface two-phase Navier--Stokes flow that satisfies a discrete energy dissipation law. We introduce an SAV reformulation to ensure the stable treatment of the surface and gravitational potential energies, together with an energy-stable ALE transfer operator that controls the density-weighted kinetic energy across moving meshes. The resulting scheme is linear and satisfies a modified energy-dissipation inequality. Numerical experiments, including rising-bubble benchmarks, demonstrate the robustness of the method and confirm second-order convergence in time and arbitrarily high-order convergence in space.
\end{abstract}


\medskip

\section{Introduction}\label{sec:introduction}
Flows of two immiscible fluids separated by a moving interface arise in many
scientific and engineering applications.  Two broad descriptions are commonly
used: diffuse-interface and sharp-interface models
\cite{GrossReusken2011}.  In a diffuse-interface model, the two fluids are
separated by a thin transition layer in which material parameters such as
density and viscosity vary rapidly but smoothly.  In a sharp-interface model,
the interface is a hypersurface of zero thickness across which these parameters
may be discontinuous.

We consider two immiscible fluids in a fixed bounded polygonal
(\(d=2\)) or polyhedral (\(d=3\)) container
\(\Om\subset\R^d\).  At time \(t\), a smooth closed
interface \(\Gam(t)\subset\Om\) separates the phase domains, so that
\(\Om=\Om_-(t)\cup\Gam(t)\cup\Om_+(t)\).  The outer boundary is decomposed
into relatively open, disjoint unions of flat boundary faces
\(\Gam_1\) and \(\Gam_2\), with
\(\partial\Om=\overline{\Gam_1}\cup\overline{\Gam_2}\).
A two-dimensional configuration is shown in \Cref{fig_domain}.
The velocity \(\bm u\) and pressure \(p\) satisfy the sharp-interface
problem
\begin{align}\label{PDE}
\left\{
\begin{aligned}
\rho_\pm\bigl\{\partial_t\bm u+(\bm u\cdot\nabla)\bm u\bigr\}
-\nabla\cdot\bm\sigma_\pm
&=\rho_\pm\bm f
&&\text{in }\Om_\pm(t),\\
\nabla\cdot\bm u&=0
&&\text{in }\Om_\pm(t),\\
[\bm u]_-^+&=\bm0
&&\text{on }\Gam(t),\\
[\bm\sigma\bm\nu]_-^+&=\gamma\kappa\bm\nu
&&\text{on }\Gam(t),\\
\bm u&=\bm0
&&\text{on }\Gam_1,\\
\bm u\cdot\bm n=0,\qquad
\bm\sigma\bm n\cdot\bm\tau&=0
&&\text{on }\Gam_2,\\
\bm u(\cdot,0)&=\bm u^0
&&\text{in }\Om.
\end{aligned}
\right.
\end{align}
Here \(\gamma\) is the surface tension coefficient.  The gravitational force is
\(\bm f=-g\bm e_d\), where \(g>0\) is the gravitational acceleration and
\(\bm e_d=(0,\ldots,0,1)^{\mathsf T}\in\mathbb R^d\) is the unit vector in
the \(x_d\)-direction.  The physical stress is
\[
 \bm\sigma_\pm
 =2\mu_\pm\bm D(\bm u)-p_\pm\bm I,
 \qquad
 \bm D(\bm u)=\frac12\bigl(\nabla\bm u+(\nabla\bm u)^\top\bigr).
\]
We write \(\rho=\rho_\pm\) and \(\mu=\mu_\pm\) in \(\Om_\pm(t)\), and
\([q]_-^+=q_+-q_-\). The unit normal \(\bm\nu\) points from
\(\Om_-(t)\) to \(\Om_+(t)\); \(\bm n\) is the outer unit normal on
\(\partial\Om\), and \(\bm\tau\) is any tangential vector on \(\Gam_2\).
The interface velocity \(\mathcal V_\Gam\) satisfies
\(\mathcal V_\Gam\cdot\bm\nu=\bm u\cdot\bm\nu\),
and we use the curvature convention
$ -\Delta_s\id=\kappa\bm\nu$.
Thus \(\kappa>0\) when \(\Om_-(t)\) is locally convex.  The tangential
velocity changes only the parametrization of \(\Gam(t)\) and may be chosen
freely.
\begin{figure}[htp]
\centering \includegraphics[width=2.15in]{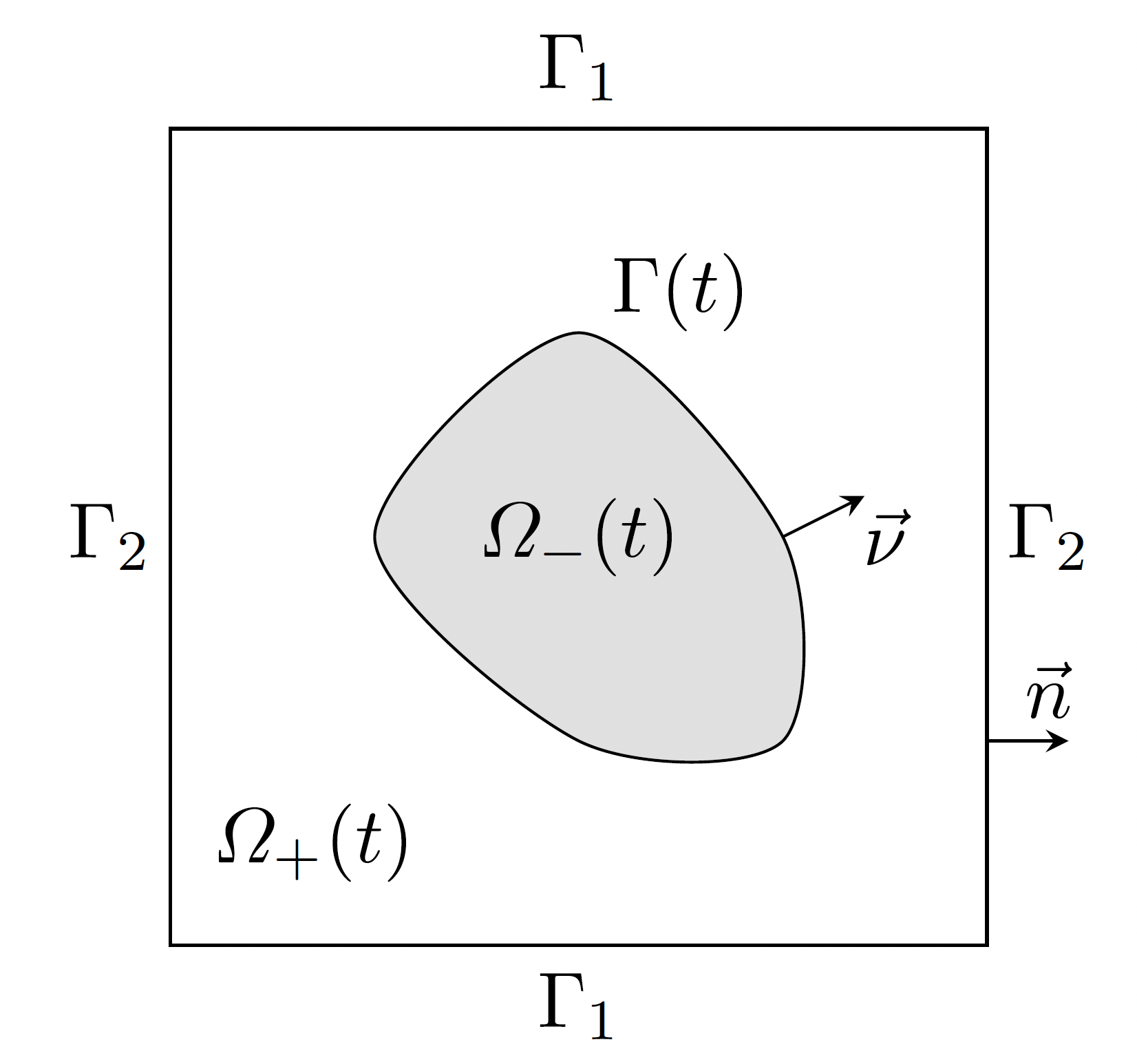}
\vspace{-4pt}
\caption{The domain \(\Omega\) in the case \(d=2\).}
\label{fig_domain}
\end{figure}

Numerical methods for \eqref{PDE} can generally be classified as Eulerian,
Lagrangian, or arbitrary Lagrangian--Eulerian (ALE) methods.  In the ALE
approach, the mesh follows the moving interface, whereas the bulk mesh velocity
need not coincide with the fluid velocity.  The interface is thus tracked
explicitly while the interior mesh can be adjusted as the flow evolves.

Existing work on fitted moving-interface methods addresses mesh motion and stability, as well as the approximation of surface tension. The ALE finite element formulation for incompressible viscous flow goes back to \cite{HughesLiuZimmermann1981}. Geometric conservation and temporal stability are studied in \cite{BoffiGastaldi2004,BadiaCodina2006,FormaggiaNobile2004,BonitoKyzaNochetto2013}; ALE methods for moving interfaces and two-phase flow include \cite{GanesanTobiska2012,Fu2020}. Harmonic mesh motion and high-order evolving or isoparametric finite element spaces are analyzed in \cite{Edelmann2022,ElliottRanner2021,LiXiaYang2023,LiMaQiu2026}. On the fitted interface, surface tension can be treated through the first variation of interfacial area. This treatment uses the tangential-gradient identity of \cite{Dziuk1991}; see also \cite{BarrettGarckeNurnberg2020}. Parametric methods for capillary and two-phase flow include \cite{Bansch2001,BarrettGarckeNurnberg2015,AgneseNurnberg2020,GarckeNurnbergZhao2023}. The continuous energy law of \eqref{PDE} gives the stability criterion for its discretization.

The two-phase Navier--Stokes flow obeys an energy-dissipation law.  Indeed,
testing the first equation in \eqref{PDE} by \(\bm u\) and using the interface
and boundary conditions, together with the
identity
\[
 \frac{\dd}{\dd t}\abs{\Gam(t)}
 =\int_{\Gam(t)} \nabla_{\Gam(t)}\id : \nabla_{\Gam(t)}\bm u \dd s
 =\int_{\Gam(t)}\kappa \bm\nu\cdot\bm u \dd s,
\]
where $\id$ is the identity map on \(\Gam(t)\), one obtains
\begin{equation}
\frac{\dd}{\dd t} \Big( \frac12\int_{\Om(t)}\rho\abs{\bm u}^2\dd x + \gamma\abs{\Gam(t)} \Big)
 +2\int_{\Om(t)}\mu\abs{\bm D(\bm u)}^2\dd x
- \int_{\Om(t)}\rho \bm f\cdot\bm u\dd x = 0.
\end{equation} 
The transport formula gives
\[
\frac{\dd}{\dd t}\int_{\Om_\pm(t)}\rho_\pm g x_d\,\dd x
=\int_{\Om_\pm(t)}\rho_\pm g \partial_t^{\bullet} x_d \,\dd x
 =-\int_{\Om_\pm(t)}\rho_\pm\bm f\cdot\bm u\,\dd x,
\]
where \(\partial_t^{\bullet}\) denotes the material derivative along the fluid
particle trajectories.
It follows that the
total energy (including the kinetic and potential energies)
\begin{equation}\label{eq:physical-energy}
 E(t)=\frac12\int_{\Om(t)}\rho\abs{\bm u}^2\dd x
                 +\gamma\abs{\Gam(t)}+G(t),
 \qquad
 G(t)=\int_{\Om(t)}\rho g x_d \dd x.
\end{equation} 
satisfies the following inequality:
\begin{equation}\label{eq:continuous-physical-law}
 \frac{\dd}{\dd t}E(t) \le 0.
\end{equation}
Hence, the total energy of the system decays as time increases, due to the viscosity of the fluid.
An analogous energy-diminishing property is desirable for the numerical scheme, especially for solutions with limited regularity, for which convergence analysis often relies on a discrete energy inequality.

To retain the energy-diminishing structure on a moving fitted mesh, both velocity transfer and interfacial work must be treated consistently. Since the velocity belongs to different finite element spaces at successive time levels, its transfer must control the density-weighted kinetic energy. The discrete capillary and gravitational terms must also agree with the geometric variation used to update the interface. For variable-density two-phase Navier--Stokes flow, \cite{DuanLiYang2022} constructed a linear, first-order energy-diminishing ALE method using a square-root Jacobian factor. This factor is closely related to the half-density composition underlying our energy-stable ALE transfer operator. Related fitted methods for free boundaries and moving contact lines are given in \cite{ZhaoRen2020,HuLeiLiTang2026,GarckeHuLei2026}; \cite{GarckeHuLei2026} also account for gravitational potential energy. Second- and higher-order structure-preserving methods have been developed for geometric flows \cite{JiangSuZhang2024JCP,MaRao2026} and sharp-interface two-phase Stokes flow \cite{GarckeTrautweinZhang2026}. The former do not include fluid coupling, while the latter omit inertia.  As far as we know, there is no higher-order fully discrete ALE method which preserves the energy
diminishing structure of the two-phase Navier--Stokes flow at the discrete level.

As a first attempt to address this issue, we propose a second-order energy-diminishing ALE--SAV finite element method for sharp-interface two-phase Navier--Stokes flow. The SAV approach introduces an auxiliary scalar for the square root of an energy functional, leading to linear schemes with a modified energy \cite{ShenXuYang2018}. Related formulations for variable-density diffuse-interface flow and relaxation of the auxiliary energy are given in \cite{YangDong2019,JiangZhangZhao2022}.
We combine Crank--Nicolson methods with degree-\(k\) isoparametric finite elements, where \(k\geq1\). We introduce an SAV reformulation for an energy-stable treatment of the surface and gravitational energies, together with an energy-stable ALE transfer operator based on projecting the half-density composition in the density-weighted product to control the kinetic energy across moving meshes. The proposed method is shown to satisfy the desired
energy inequality under moving mesh, and the numerical results indicate approximately
second-order convergence in time and convergence of arbitrarily high order in space. Several numerical
experiments are provided to illustrate the performance of the method on benchmark examples.

The rest of this article is organized as follows.
The continuous ALE--SAV reformulation and its energy law are developed in
Section~\ref{sec:model}.  The fully discrete method and its modified-energy
inequality are presented in
Section~\ref{sec:fully-discrete}.  Implementation details and numerical experiments
are reported in Section~\ref{sec:numerics}.

\section{The ALE--SAV formulation of the sharp-interface model}\label{sec:model}

\subsection{Geometric energy and SAV reformulation}\label{sec:continuous-sav}
Let \(\Gam^0=\Gam(0)\), and represent the moving interface by
\[
 \bm X(\cdot,t):\Gam^0\to\R^d,
 \qquad \Gam(t)=\Gam[\bm X(t)]:=\bm X(\Gam^0,t),
 \qquad \bm X^0:=\bm X(\cdot,0)=\id_{\Gam^0}.
\]
A variation \(\bm\eta\) of \(\bm X\) on \(\Gam^0\) induces the field \(\bm\eta\circ\bm X^{-1}\) on \(\Gam[\bm X]\).

The geometric part of \eqref{eq:physical-energy} is
\begin{equation}\label{eq:total-geometric-energy}
 \F(\bm X)=\gamma\abs{\Gam[\bm X]}+G(\bm X).
\end{equation}
To obtain a quadratic geometric term in the time-discrete energy,
choose \(C_0>0\) sufficiently large such that
\(\F(\bm X)+C_0>0\), and define
\begin{equation}\label{eq:continuous-sav-definitions}
 R(\bm X)=\sqrt{\F(\bm X)+C_0}.
\end{equation}
Here \(R(\bm X)\) is determined by the interface; in the reformulation,
\(r=r(t)\) is an independent scalar unknown.

To specify its evolution, we calculate \(D\F\).  Writing
\(G(\bm X):=G(\Gam[\bm X])\), we have
\[
G(\bm X)=\rho_+\int_\Om g x_d\dd x
 +(\rho_--\rho_+)\int_{\Om_-[\bm X]}g x_d\dd x.
\]
Hence, the shape derivative in the direction \(\bm\eta\) is
\begin{equation}\label{eq:potential-shape-derivative}
 DG(\bm X)[\bm\eta]= (\rho_--\rho_+)
   \int_{\Gam[\bm X]}g x_d
   (\bm\eta\circ\bm X^{-1})\cdot\bm\nu\dd s.
\end{equation}
Together with the surface variation, this gives
\begin{equation}\label{eq:geometric-first-variation}
 D\F(\bm X)[\bm\eta]=\int_{\Gam[\bm X]}
 \bigl\{\gamma\kappa+(\rho_--\rho_+)g x_d\bigr\}
 (\bm\eta\circ\bm X^{-1})\cdot\bm\nu\dd s.
\end{equation}
Set \(L_{\bm X}[\bm\eta]:=D\F(\bm X)[\bm\eta]/R(\bm X)\).
Along a smooth interface path, the chain rule gives
\[
 \frac{\dd}{\dd t}R(\bm X(t))
 =\frac{D\F(\bm X(t))[\bm X_t]}{2R(\bm X(t))}
 =\frac12L_{\bm X(t)}[\bm X_t].
\]
This motivates treating \(r=r(t)\) as an independent scalar
unknown and imposing
\(r_t=\frac12L_{\bm X}[\bm X_t]\) and
\(r(0)=R(\bm X^0)\).

The gravitational force can be absorbed into the pressure.  Since
\(\bm f=-\nabla(gx_d)\), set \(\pi_\pm=p_\pm+\rho_\pm g x_d\), with
\[
 \bm\sigma_\pm^{\rm sh}=2\mu_\pm\bm D(\bm u)-\pi_\pm\bm I
 =\bm\sigma_\pm-\rho_\pm g x_d\bm I.
\]
After fixing the common additive constant by \(\int_\Om\pi\dd x=0\),
we henceforth write \(p\) and \(\bm\sigma\) for the shifted pressure
\(\pi\) and stress \(\bm\sigma^{\rm sh}\), respectively.  The interface
jump becomes
\begin{equation}\label{eq:shifted-stress-jump}
 [\bm\sigma\bm\nu]_-^+
 =\bigl\{\gamma\kappa+(\rho_--\rho_+)g x_d\bigr\}\bm\nu.
\end{equation}
Since the work of this shifted interfacial traction is
\(D\F(\bm X)[\bm v\circ\bm X]\), the SAV reformulation replaces it by
\(r(t)L_{\bm X}[\bm v\circ\bm X]\), thereby multiplying the traction by
\(r(t)/R(\bm X(t))\).  The shifted bulk equations remain unchanged.  We
choose the material parametrization \(\bm X_t=\bm u\circ\bm X\); its
tangential component affects only the parametrization of the interface.
With the outer boundary conditions and initial velocity from \eqref{PDE},
the reformulated system is
\begin{equation}\label{eq:pre-ale-sav-system}
\left\{
\begin{aligned}
 \rho_\pm\bigl\{\partial_t\bm u+(\bm u\cdot\nabla)\bm u\bigr\}
 -\nabla\cdot\bm\sigma&=\bm0
 &&\text{in }\Om_\pm(t),
 \\
 \Div\bm u&=0
 &&\text{in }\Om_\pm(t),\\
 [\bm u]_-^+&=\bm0
 &&\text{on }\Gam(t),\\
 [\bm\sigma\bm\nu]_-^+
 &=\frac{r(t)}{R(\bm X(t))}
   \bigl\{\gamma\kappa+(\rho_--\rho_+)g x_d\bigr\}\bm\nu
 &&\text{on }\Gam(t),
 \\
 \bm X_t(\xi,t)&=\bm u(\bm X(\xi,t),t)
 &&\text{for }\xi\in\Gam^0,
 \\
 r_t&=\frac12L_{\bm X}[\bm X_t],
 \qquad
 r(0)=R(\bm X^0).
 \end{aligned}
\right.
\end{equation}

\subsection{The ALE formulation}
Let \(\Om_\pm^0=\Om_\pm(0)\) and
\(\Om^0=\Om_-^0\cup\Gam^0\cup\Om_+^0=\Om\).
Under the material parametrization in
\eqref{eq:pre-ale-sav-system}, the interface velocity is
\[
 \mathcal V_\Gam=\bm X_t\circ\bm X^{-1}
 =\bm u|_{\Gam(t)}.
\]
A mesh velocity \(\bm w\) with trace \(\mathcal V_\Gam\) generates the ALE
flow map through
\begin{equation}\label{eq:continuous-ale-flow-map}
\left\{
\begin{aligned}
 \partial_t\bm\phi(\bm x^0,t)
 &=\bm w(\bm\phi(\bm x^0,t),t) &&\text{in }\Om^0,\\
 \bm\phi(\bm x^0,0)&=\bm x^0 &&\text{in }\Om^0,\\
 \bm\phi(\bm X^0(\xi),t)&=\bm X(\xi,t)
 &&\text{for }\xi\in\Gam^0,
\end{aligned}
\right.
\end{equation}
where we choose \(\bm w\) as the phasewise harmonic extension of
\(\mathcal V_\Gam\) by
\begin{equation}\label{eq:continuous-harmonic}
\left\{
\begin{aligned}
 -\Delta\bm w^\pm&=\bm0 &&\text{in }\Om_\pm(t),\\
 \bm w^\pm&=\mathcal V_\Gam &&\text{on }\Gam(t),\\
 \bm w&=\bm0 &&\text{on }\Gam_1,\\
 \bm w\cdot\bm n&=0,\qquad
 (\nabla\bm w)\bm n\cdot\bm\tau=0 &&\text{on }\Gam_2.
\end{aligned}
\right.
\end{equation}
Thus a no-slip wall fixes the mesh pointwise, whereas the homogeneous
free-slip conditions keep \(\Gam_2\) fixed as a geometric wall while
allowing tangential mesh motion.
The harmonic extension is a mesh-motion choice and
contributes no physical dissipation.
We assume that
\(\bm\phi(\cdot,t):\Om_\pm^0\to\Om_\pm(t)\) is bi-Lipschitz and
orientation preserving, with
\[
 \bm\phi(\Om_\pm^0,t)=\Om_\pm(t),\qquad
 \rho(\bm\phi(\bm x^0,t),t)=\rho^0(\bm x^0),\qquad
 J(\bm x^0,t)=\det\nabla_{\bm x^0}\bm\phi(\bm x^0,t)>0.
\]
The boundary partition is preserved: \(\bm\phi|_{\Gam_1}=\id\) and
\(\bm\phi(\Gam_2,t)=\Gam_2\).
Here \(\rho^0=\rho(\cdot,0)\).

The ALE derivative and relative velocity are
\[
\partial_t^{\bullet}\bm u
 =\partial_t\bm u+(\bm w\cdot\nabla)\bm u,
 \qquad
 \bm\beta=\bm u-\bm w.
\]
Thus the material acceleration is
\(\partial_t\bm u+(\bm u\cdot\nabla)\bm u
=\partial_t^{\bullet}\bm u+(\bm\beta\cdot\nabla)\bm u\).
To recover the kinetic-energy chain rule, introduce the half-density
ALE derivative
\begin{equation}\label{eq:half-density-derivative}
\mathcal D_t^{\bullet}\bm u
 =\partial_t^{\bullet} \bm u
  +\frac12(\Div\bm w)\bm u
 =\left[J^{-1/2}\partial_t
   \left(J^{1/2}(\bm u\circ\bm\phi)\right)\right]
   \circ\bm\phi^{-1}.
\end{equation}
Phase preservation gives the kinetic chain rule
\begin{equation}\label{eq:continuous-kinetic-chain}
 \int_\Om\rho\,\mathcal D_t^{\bullet}\bm u\cdot\bm u\dd x
 =\frac12\frac{\dd}{\dd t}
   \int_\Om\rho|\bm u|^2\dd x.
\end{equation}

\subsection{The ALE--SAV weak formulation}\label{sec:continuous-weak}

By \eqref{eq:half-density-derivative}, the material
acceleration in \eqref{eq:pre-ale-sav-system} can be written as
\(\mathcal D_t^\bullet\bm u+(\bm\beta\cdot\nabla)\bm u
-\tfrac12(\Div\bm w)\bm u\).  We use this representation to formulate the
weak problem.  Define
\begin{align*}
 \mathbf V
 &=\bigl\{\bm v\in H^1(\Om)^d:
       \bm v=\bm0\ \text{on }\Gam_1,\ 
       \bm v\cdot\bm n=0\ \text{on }\Gam_2\bigr\},\\
 P&=L_0^2(\Om),\\
 \mathbf W(t;\bm v_\Gam)
 &=\bigl\{\bm z:\bm z^\pm\in H^1(\Om_\pm(t))^d,
       \ \bm z^\pm=\bm v_\Gam\ \text{on }\Gam(t),\\
 &\hspace{5.5em}\bm z=\bm0\ \text{on }\Gam_1,
       \ \bm z\cdot\bm n=0\ \text{on }\Gam_2\bigr\},\\
 \mathbf W_0(t)&=\mathbf W(t;\bm0).
\end{align*}
For a bulk domain \(D\), let \((\cdot,\cdot)_D\) and \(\norm{\cdot}_D\)
denote the \(L^2(D)\) inner product and norm, respectively.
We use \(\langle\cdot,\cdot\rangle_{\Gam^0}\) for the
\(L^2(\Gam^0)\) inner product on the reference interface.

The material interface choice gives \(\bm\beta=\bm0\) on
\(\Gam(t)\) and \(\Gam_1\), while \(\bm\beta\cdot\bm n=0\) on \(\Gam_2\).
Using \(\Div\bm u=0\), phasewise integration by parts gives
\begin{equation}\label{eq:continuous-skew-equivalence}
 \frac12\bigl\{(\rho(\bm\beta\cdot\nabla)\bm u,\bm v)_\Om
 -(\rho(\bm\beta\cdot\nabla)\bm v,\bm u)_\Om\bigr\}
 =\bigl(\rho\{(\bm\beta\cdot\nabla)\bm u
       -\tfrac12(\Div\bm w)\bm u\},\bm v\bigr)_\Om.
\end{equation}
The left-hand side is skew-symmetric in \(\bm u\) and
\(\bm v\), and therefore vanishes for \(\bm v=\bm u\) without an
additional divergence constraint.

The weak ALE--SAV formulation seeks
\(\bm u(t)\in\mathbf V\), \(p(t)\in P\),
\(\bm w(t)\in\mathbf W(t;\bm u|_{\Gam(t)})\), and sufficiently regular
\((\bm X(t),r(t),\bm\phi(t))\) such that, for all
\(\bm v\in\mathbf V\), \(q\in P\),
\(\bm\eta\in L^2(\Gam^0)^d\), \(s\in\R\),
\(\bm z\in\mathbf W_0(t)\), and \(\bm\psi\in L^2(\Om^0)^d\),
\begin{subequations}\label{eq:continuous-sav-system}
\begin{align}
 \begin{aligned}[b]
 &(\rho\mathcal D_t^{\bullet}\bm u,\bm v)_\Om +\frac12\bigl\{(\rho(\bm\beta\cdot\nabla)\bm u,\bm v)_\Om
   -(\rho(\bm\beta\cdot\nabla)\bm v,\bm u)_\Om\bigr\}\\
 &\quad+2(\mu\bm D(\bm u),\bm D(\bm v))_\Om
  -(p,\Div\bm v)_\Om +rL_{\bm X}[\bm v\circ\bm X]
 \end{aligned}
 &=0,\label{eq:continuous-sav-momentum}\\
 (\Div\bm u,q)_\Om&=0,\\
 \left\langle\bm X_t-\bm u\circ\bm X,\bm\eta\right\rangle_{\Gam^0}
 &=0,\\
 \left(r_t-\frac12L_{\bm X}[\bm X_t]\right)s&=0,\\
 \sum_{\pm}(\nabla\bm w^\pm,\nabla\bm z^\pm)_{\Om_\pm(t)}&=0,
 \label{eq:continuous-sav-harmonic-weak}\\
 \left(\partial_t\bm\phi-\bm w\circ\bm\phi,
       \bm\psi\right)_{\Om^0}&=0.
\end{align}
\end{subequations}
The weak problem is supplemented with the initial data and
interface compatibility conditions
\[
\begin{alignedat}{2}
 \bm u(\cdot,0)&=\bm u^0,
 &\qquad \bm X(\cdot,0)&=\bm X^0=\id_{\Gam^0},\\
 r(0)&=R(\bm X^0),
 &\qquad \bm\phi(\cdot,0)&=\id_{\Om^0},\\
 \bm\phi(\bm X^0(\xi),t)&=\bm X(\xi,t),
 &\qquad&\xi\in\Gam^0.
\end{alignedat}
\]

\subsection{Equivalence and continuous energy law}
\begin{lemma}\label{thm:continuous-sav}
Every sufficiently regular solution of
\eqref{eq:continuous-sav-system} satisfies
\(r(t)=R(\bm X(t))\).  Hence, the SAV system is equivalent to
the shifted sharp-interface model equipped with the material
parametrization and the ALE construction above; conversely, any sufficiently
regular solution of this augmented model satisfies the SAV
system upon setting \(r=R(\bm X)\).  Moreover,
\begin{equation}\label{eq:continuous-modified-law}
 \frac{\dd}{\dd t}
 \left\{\frac12\int_\Om\rho\abs{\bm u}^2\dd x
 +r^2-C_0\right\}
 +2\int_\Om\mu\abs{\bm D(\bm u)}^2\dd x
 =0.
\end{equation}
Since \(r=R(\bm X)\), the energy in braces is \(E(t)\),
and hence \(E(t)\) is non-increasing.
\end{lemma}

\begin{proof}
The scalar equation and the chain rule give
\begin{equation}\label{eq:continuous-sav-consistency}
\begin{aligned}
 \frac{\dd}{\dd t}\bigl(r(t)-R(\bm X(t))\bigr)
 &=r_t-\frac{D\F(\bm X)[\bm X_t]}{2R(\bm X)}=\frac12L_{\bm X}[\bm X_t]
  -\frac12L_{\bm X}[\bm X_t]=0,\\
 r(0)-R(\bm X^0)&=0.
\end{aligned}
\end{equation}
Hence, for every interface variation \(\bm\eta\),
\begin{equation}\label{eq:continuous-sav-identities}
\begin{aligned}
 r(t)=R(\bm X(t)), \qquad \frac{r(t)}{R(\bm X(t))}=1, \qquad  r(t)L_{\bm X}[\bm\eta]&=D\F(\bm X)[\bm\eta].
\end{aligned}
\end{equation}
The same identities give the converse upon setting \(r=R(\bm X)\) in
the shifted sharp-interface model.

Test \eqref{eq:continuous-sav-system} with \(\bm v=\bm u\) and \(q=p\).
The skew terms cancel, and \((\Div\bm u,p)_\Om=0\). By
\eqref{eq:continuous-kinetic-chain},
\begin{equation}\label{eq:continuous-momentum-balance}
\begin{aligned}
 \frac12\frac{\dd}{\dd t}(\rho\bm u,\bm u)_\Om
 +2(\mu\bm D(\bm u),\bm D(\bm u))_\Om +rL_{\bm X}[\bm u\circ\bm X]=0.
\end{aligned}
\end{equation}
The kinematic and scalar equations yield
\begin{align}
 rL_{\bm X}[\bm u\circ\bm X]
 &=rL_{\bm X}[\bm X_t] =2rr_t=\frac{\dd}{\dd t}r^2,
 \label{eq:continuous-sav-work}\\
 r^2-C_0
 &=\F(\bm X)
 =\gamma\abs{\Gam[\bm X]}+G(\bm X).
 \label{eq:continuous-geometric-identity}
\end{align}
Substituting
\eqref{eq:continuous-sav-work} and
\eqref{eq:continuous-geometric-identity} into
\eqref{eq:continuous-momentum-balance} proves
\eqref{eq:continuous-modified-law}.
\end{proof}

\section{The fully discrete ALE--SAV method}\label{sec:fully-discrete}

\subsection{Evolving isoparametric finite element spaces}

Fix an integer \(k\ge1\).  Let \(\K_h^0\) be a shape-regular and
quasi-uniform fitted triangulation
of the initial phase partition, with interface \(\Gam_h^0\).  Let
$\x^0=(\bm\xi_1,\ldots,\bm\xi_M)\in\R^{dM}$
collect all degree-\(k\) geometry nodes of \(\K_h^0\).  A nodal vector
\(\x=(\bm x_1,\ldots,\bm x_M)\) determines the fitted triangulation
\(\K_h[\x]\), its phase subdomains \(\Om_{h,\pm}[\x]\), and its interface
\(\Gam_h[\x]\).  Every element is a degree-\(k\) isoparametric image of a
reference simplex \(\widehat K\):
\begin{equation}\label{eq:isoparametric-element}
 K[\x]=F_K[\x](\widehat K),\qquad
 F_K[\x]\in[\mathbb P_k(\widehat K)]^d,\qquad
 \det DF_K[\x](\widehat x)>0
 \quad\forall\widehat x\in\widehat K.
\end{equation}
The element maps agree on common faces.  Interface faces are curved
isoparametric faces and approximate a smooth interface with geometric error
of order \(O(h^{k+1})\); see
\cite{ElliottRanner2021,Lenoir-1986}.  The fixed outer boundary is represented
exactly and consists of planar faces.  We write
\[
 \Om_h[\x]=\Om_{h,-}[\x]\cup\Gam_h[\x]\cup\Om_{h,+}[\x].
\]
An admissible configuration satisfies \eqref{eq:isoparametric-element},
preserves the phase labels, fixes \(\Gam_1\) pointwise, and maps
\(\Gam_2\) onto itself.

Define the mapped scalar space
\begin{equation}\label{eq:mapped-scalar-space}
 S_h^k[\x]
 :=\left\{v_h\in C^0(\overline\Om):
 v_h\circ F_K[\x]\in\mathbb P_k(\widehat K)
 \quad\forall K[\x]\in\K_h[\x]\right\}.
\end{equation}
Here \(S_h^{k+1}[\x]\) is defined analogously, with \(k\) replaced by
\(k+1\).  The fluid velocity and phasewise pressure spaces are
\begin{align}
 X_h[\x]
 &=[S_h^{k+1}[\x]]^d\cap\mathbf V,
 \label{eq:high-order-velocity-space}\\
 M_h[\x]
 &=\left(
 S_h^k[\x]|_{\Om_{h,+}[\x]}
 \times S_h^k[\x]|_{\Om_{h,-}[\x]}
 \right)\cap L_0^2(\Om).
 \label{eq:high-order-pressure-space}
\end{align}
Thus the pressure may jump across the fitted interface.  Moreover,
\begin{equation}\label{eq:shifted-pressure-inclusion}
 x_d\circ F_K[\x]=(F_K[\x])_d\in\mathbb P_k(\widehat K),
\end{equation}
so the pressure shift \(\rho g x_d\), after subtraction of its global
mean, belongs to the phasewise pressure space.  The geometry and
mesh-velocity space, its homogeneous-interface subspace, and the scalar
interface trace space are
\begin{align}
 W_h[\x]
 &=\left\{\bm z_h\in[S_h^k[\x]]^d:
 \bm z_h=\bm0\text{ on }\Gam_1,\quad
 \bm z_h\cdot\bm n=0\text{ on }\Gam_2\right\},
 \label{eq:high-order-mesh-space}\\
 \mathring W_h[\x]
 &=\left\{\bm z_h\in W_h[\x]:
 \bm z_h=\bm0\text{ on }\Gam_h[\x]\right\},\\
 \mathcal S_h[\x]
 &=\operatorname{Tr}_{\Gam_h[\x]}S_h^k[\x].
 \label{eq:high-order-interface-space}
\end{align}
It is known that the Taylor--Hood pair satisfies the discrete
inf--sup estimate on uniformly regular fitted mesh families:
\begin{equation}\label{eq:discrete-inf-sup}
 \|q_h\|_{L^2(\Om)}
 \le C\sup_{\bm0\ne\bm v_h\in X_h[\x]}
 \frac{|(q_h,\Div\bm v_h)_\Om|}{\|\bm v_h\|_{H^1(\Om)}}
 \qquad\forall q_h\in M_h[\x],
\end{equation}
where \(C\) is independent of \(h\) within such a family.

Let \(\x(t)=(\bm x_1(t),\ldots,\bm x_M(t))\), with \(\x(0)=\x^0\),
denote the evolving nodal configuration.  For the fully discrete method,
let \(t_n=n\tau\) and
\(t_{n+\frac12}=t_n+\tau/2\), and write
\[
 \begin{aligned}
  \x^n&=\x(t_n),\qquad
  \bm x_j^n=\bm x_j(t_n),\qquad
  \K_h^n=\K_h[\x^n],\\
  \Om_h^n&=\Om_h[\x^n],\qquad
  \Om_{h,\pm}^n=\Om_{h,\pm}[\x^n],\qquad
  \Gam_h^n=\Gam_h[\x^n].
 \end{aligned}
\]
Let \(\rho_h^n|_{\Om_{h,\pm}^n}=\rho_\pm\) and
\(\mu_h^n|_{\Om_{h,\pm}^n}=\mu_\pm\).
The notation at other time levels is analogous.

\subsection{Energy-stable ALE transfer operator}
Fix two mesh times \(s\) and \(t\) in one fixed-connectivity interval.  Write
\(\Om_h(r)=\Om_h[\x(r)]\) and
\(\rho_h(r)|_{\Om_{h,\pm}[\x(r)]}=\rho_\pm\) for \(r=s,t\).
The relative ALE map \(\bm\phi_{h,s\to t}\) takes
\(F_K[\x(s)](\widehat x)\) to \(F_K[\x(t)](\widehat x)\) on corresponding
elements and preserves their phase labels.  Its Jacobian is
\begin{equation}\label{eq:relative-ale-jacobian}
 J_{s\to t}(F_K[\x(s)](\widehat x))
 :=\det D\bm\phi_{h,s\to t}(F_K[\x(s)](\widehat x))
 =\frac{\det DF_K[\x(t)](\widehat x)}
        {\det DF_K[\x(s)](\widehat x)}>0,
 \qquad \widehat x\in\widehat K.
\end{equation}
Ordinary ALE composition changes the density-weighted kinetic norm through
this Jacobian.  We therefore define the half-density composition by
\begin{equation}\label{eq:half-density-composition}
 (\C_{s\to t}\bm v)(F_K[\x(t)](\widehat x))
 =J_{s\to t}(F_K[\x(s)](\widehat x))^{-1/2}
   \bm v(F_K[\x(s)](\widehat x)),
 \qquad \widehat x\in\widehat K.
\end{equation}
The field \(\C_{s\to t}\bm v_h\) need not belong to the target space
\(X_h[\x(t)]\).  Let
\(\Pi_h(t):L^2(\Om_h(t))^d\to X_h[\x(t)]\) be the density-weighted
orthogonal projection, defined by
\[
 (\rho_h(t)\Pi_h(t)\bm z,\vh)_{\Om_h(t)}
 =(\rho_h(t)\bm z,\vh)_{\Om_h(t)}
 \qquad\forall\vh\in X_h[\x(t)].
\]
We then define the ALE transfer operator by
\begin{equation}\label{eq:projected-transfer}
 \mathcal T_{s\to t}=\Pi_h(t)\C_{s\to t}.
\end{equation}
The mesh correspondence and the associated velocity transfer are illustrated
in \Cref{fig:ale-transfer}.
\begin{figure}[t]
 \centering
 \includegraphics[width=0.86\textwidth]{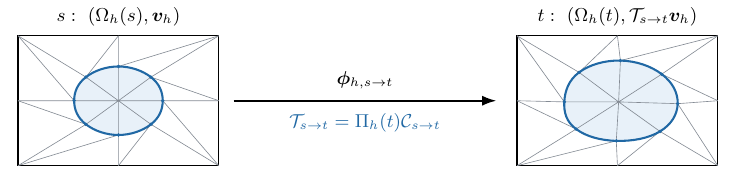}
 \caption{ALE map and velocity transfer.}
 \label{fig:ale-transfer}
\end{figure}

\begin{lemma}[Energy stability of the ALE transfer]
\label{lem:ale-transfer-stability}
For every \(\bm v_h\in X_h[\x(s)]\),
\begin{equation}\label{eq:ale-transfer-stability}
 \norm{\sqrt{\rho_h(t)}\,\mathcal T_{s\to t}\bm v_h}_{\Om_h(t)}^2
 \le \norm{\sqrt{\rho_h(s)}\,\bm v_h}_{\Om_h(s)}^2.
\end{equation}
\end{lemma}

\begin{proof}
Because the ALE map preserves the phase labels, corresponding elements have
the same phase density, denoted by \(\rho_K\).  An elementwise change of
variables shows that the half-density composition is an isometry:
\begin{align*}
 \norm{\sqrt{\rho_h(t)}\,\C_{s\to t}\bm v_h}_{\Om_h(t)}^2
 &=\sum_K\int_{K[\x(s)]}\rho_K
   J_{s\to t}^{-1}\abs{\bm v_h}^2
   J_{s\to t}\dd x\\
 &=\sum_K\int_{K[\x(s)]}\rho_K\abs{\bm v_h}^2\dd x
 =\norm{\sqrt{\rho_h(s)}\,\bm v_h}_{\Om_h(s)}^2.
\end{align*}
Orthogonality and Cauchy--Schwarz show that the weighted projection is
nonexpansive:
\[
 \norm{\sqrt{\rho_h(t)}\,\Pi_h(t)\bm z}_{\Om_h(t)}^2
 =(\rho_h(t)\bm z,\Pi_h(t)\bm z)_{\Om_h(t)}
 \le\norm{\sqrt{\rho_h(t)}\,\bm z}_{\Om_h(t)}
       \norm{\sqrt{\rho_h(t)}\,\Pi_h(t)\bm z}_{\Om_h(t)}.
\]
Taking \(\bm z=\C_{s\to t}\bm v_h\) and combining the two estimates
yields
\[
 \norm{\sqrt{\rho_h(t)}\,\mathcal T_{s\to t}\bm v_h}_{\Om_h(t)}^2
 \le\norm{\sqrt{\rho_h(t)}\,\C_{s\to t}\bm v_h}_{\Om_h(t)}^2
 =\norm{\sqrt{\rho_h(s)}\,\bm v_h}_{\Om_h(s)}^2.
\]
The proof is complete.
\end{proof}

\subsection{Fully discrete Crank--Nicolson scheme}

At \(t_n\), the trace projection
\(\Rop_h^n:X_h[\x^n]\to(\mathcal S_h[\x^n])^d\) is defined by
\begin{equation}\label{eq:trace-projection}
 (\Rop_h^n\bm v,\bm\eta)_{\Gam_h^n}
 =(\bm v,\bm\eta)_{\Gam_h^n}
 \qquad\forall\bm\eta\in(\mathcal S_h[\x^n])^d.
\end{equation}
The projection at the other time levels is defined analogously.

The corresponding forms are
\begin{subequations}\label{eq:configuration-forms}
\begin{align}
 a_h^n(\bm v,\bm z)
 &=2(\mu_h^n\bm D(\bm v),\bm D(\bm z))_{\Om_h^n},\\
 b_h^n(\bm v,q)&=-(q,\Div\bm v)_{\Om_h^n},\\
 c_h^n(\bm\beta;\bm v,\bm z)
 &=\frac12\bigl(\rho_h^n(\bm\beta\cdot\nabla)\bm v,\bm z\bigr)_{\Om_h^n}
  -\frac12\bigl(\rho_h^n(\bm\beta\cdot\nabla)\bm z,\bm v\bigr)_{\Om_h^n}.
\end{align}
\end{subequations}
The forms at the other time levels are defined analogously.

For a nondegenerate interface map \(\bm X_h\), define the
discrete total geometric energy by
\begin{align}
 \F_h(\bm X_h)
 &=\gamma\abs{\Gam_h[\bm X_h]}+G_h(\bm X_h),
 \label{eq:discrete-geometric-energy}\\
 G_h(\bm X_h)
 &=\int_\Om\rho_h(\bm X_h)g x_d\dd x,
 \qquad
r_h(\bm X_h)=\sqrt{\F_h(\bm X_h)+C_0}.
 \label{eq:discrete-potential-energy}
\end{align}
The surface term and its variation use the same surface
quadrature.  The volume quadrature for \(G_h\) is exact for
\((F_K)_d\det DF_K\) on every reference element; degree of exactness
\((d+1)k-d\) suffices.  Thus \(G_h\) is the exact phasewise potential
and depends only on the interface.
The area variation, assembled with the chosen surface quadrature, is
\begin{align}
 D\{\gamma\abs{\Gam_h}\}[\bm\eta_h]
 &=\gamma\int_{\Gam_h}
   \bm P_h:\nabla_s\bm\eta_h\dd s,
 \qquad
 \bm P_h=\bm I-\bm\nu_h\otimes\bm\nu_h,
 \label{eq:discrete-surface-variation}
\end{align}
With fixed material labels and zero normal displacement on the outer boundary,
the potential variation is
\begin{align}
 DG_h(\bm X_h)[\bm\eta_h]
 &=(\rho_--\rho_+)\int_{\Gam_h}
   g x_d\bm\eta_h\cdot\bm\nu_h\dd s.
 \label{eq:discrete-potential-variation}
\end{align}
Here \(\bm\nu_h\) points from \(\Om_{h,-}\) to \(\Om_{h,+}\).
The surface quadrature is also exact for the gravitational
integrand in \eqref{eq:discrete-potential-variation}.  The trace projection
and kinematic equation use the same surface \(L^2\) product.

The interface maps \(\Xh^n:\Gam_h^0\to\Gam_h^n\) use common material
labels at all time levels.  
Given the initial fitted interface and velocity, initialize
\begin{equation}\label{eq:discrete-sav-initialization}
 r^0=r_h(\Xh^0).
\end{equation}
The scalar update need not preserve \(r^n=r_h(\Xh^n)\) at later time levels.

\paragraph{Step 1: Construct the midpoint mesh from known data.}
Predict the midpoint interface by
\begin{equation}\label{eq:discrete-carrier}
 \Xh^{n+\frac12}(\zeta)=\Xh^n(\zeta)
 +\frac\tau2(\Rop_h^n\uh^n)(\Xh^n(\zeta)),
 \qquad \zeta\in\Gam_h^0.
\end{equation}
Compute the current harmonic mesh velocity
\(\bm w_h^n\in W_h[\x^n]\) and update all degree-\(k\) geometry nodes:
\begin{subequations}\label{eq:discrete-carrier-flow-map}
\begin{align}
 \bm w_h^n|_{\Gam_h^n}&=\Rop_h^n\uh^n,\\
 \sum_{\pm}\int_{\Om_{h,\pm}^n}
 \nabla\bm w_h^n:\nabla\bm z_h\dd x&=0
 &&\forall\bm z_h\in\mathring W_h[\x^n],\\
 \bm x_j^{n+\frac12}&=\bm x_j^n
       +\frac\tau2\bm w_h^n(\bm x_j^n),\\
 \bm\phi_h^{n+\frac12}(\bm x_j^0)&=\bm x_j^{n+\frac12}.
\end{align}
\end{subequations}
The interface nodes satisfy \eqref{eq:discrete-carrier}.
On the midpoint mesh, set
\begin{equation}\label{eq:discrete-surface-functional}
 L_h^{n+\frac12}[\bm\eta_h]
 =\frac{D\F_h(\Xh^{n+\frac12})[\bm\eta_h]}
 {r_h(\Xh^{n+\frac12})}.
\end{equation}
Set
\begin{equation}\label{eq:incoming-transfer}
 \widehat{\bm u}_h^{\,n}
 =\mathcal T_{t_n\to t_{n+\frac12}}\uh^n.
\end{equation}
The frozen advector is initialized from \(\uh^0\) and subsequently
extrapolated from two time levels:
\begin{equation}\label{eq:discrete-advector}
 \bm u_{\mathrm{adv},h}^{n+\frac12}
 =\begin{cases}
 \mathcal T_{t_0\to t_{\frac12}}\uh^0,
   & n=0,\\[2mm]
 \dfrac32\mathcal T_{t_n\to t_{n+\frac12}}\uh^n
 -\dfrac12\mathcal T_{t_{n-1}\to t_{n+\frac12}}\uh^{n-1},
   & n\ge1.
 \end{cases}
\end{equation}
The first step thus requires no velocity at \(t_{-1}\) and uses the same
coupled solve and endpoint update as subsequent steps.  For smooth solutions,
this one-time startup gives an \(O(\tau^2)\) starting error and does not reduce
the formal second-order temporal accuracy.
The frozen convective harmonic velocity
\(\wh^{n+\frac12}
\in W_h[\x^{n+\frac12}]\) satisfies
\begin{subequations}\label{eq:discrete-mesh-velocity}
\begin{align}
 \wh^{n+\frac12}|_{\Gam_h^{n+\frac12}}
 &=\Rop_h^{n+\frac12}
   \bm u_{\mathrm{adv},h}^{n+\frac12},\\
 \sum_{\pm}\int_{\Om_{h,\pm}^{n+\frac12}}
 \nabla\wh^{n+\frac12}:\nabla\bm z_h\dd x&=0
 &&\forall\bm z_h\in\mathring W_h[\x^{n+\frac12}],\\
 \betah^{n+\frac12}
 &=\bm u_{\mathrm{adv},h}^{n+\frac12}
   -\wh^{n+\frac12}.
\end{align}
\end{subequations}
The predictor \(\bm w_h^n\) constructs the midpoint mesh, whereas
\(\wh^{n+\frac12}\) enters only the frozen convective coefficient
\(\betah^{n+\frac12}\).

\paragraph{Step 2: Solve the coupled system on the midpoint mesh.}
The unknowns are
\[
 (\widehat{\bm u}_h^{\,n+1},p_h^{n+1/2},\Xh^{n+1},r^{n+1})
 \in X_h[\x^{n+\frac12}]
 \times M_h[\x^{n+\frac12}]
 \times(\mathcal S_h[\x^{n+\frac12}])^d\times\R.
\]
Here \(p_h^{n+1/2}\) approximates the shifted pressure defined in
\Cref{sec:continuous-sav}.
Write
\begin{equation}\label{eq:discrete-midpoints}
 \bm u_h^{n+1/2}
 =\frac12(\widehat{\bm u}_h^{\,n+1}+\widehat{\bm u}_h^{\,n}),
 \qquad
 r^{\,n+1/2}=\frac12(r^{n+1}+r^n).
\end{equation}
The first three equations below hold for every test triple
\[
 (\vh,q_h,\etah)\in
 X_h[\x^{n+\frac12}]
 \times M_h[\x^{n+\frac12}]
 \times(\mathcal S_h[\x^{n+\frac12}])^d.
\]
\begin{subequations}\label{eq:fully-discrete-scheme}
\begin{align}
 \left(\rho_h^{n+\frac12}
 \frac{\widehat{\bm u}_h^{\,n+1}-\widehat{\bm u}_h^{\,n}}{\tau},\vh\right)_{\Om_h^{n+\frac12}}
 &+c_h^{n+\frac12}
   (\betah^{n+\frac12};\bm u_h^{n+1/2},\vh)
 +a_h^{n+\frac12}(\bm u_h^{n+1/2},\vh)
 \notag\\
 &+b_h^{n+\frac12}(\vh,p_h^{n+1/2})
 +r^{\,n+1/2}L_h^{n+\frac12}
  [\Rop_h^{n+\frac12}\vh]
=0,
 \label{eq:discrete-momentum}\\
 b_h^{n+\frac12}(\bm u_h^{n+1/2},q_h)&=0,
 \label{eq:discrete-divergence}\\
 \left(\frac{\Xh^{n+1}-\Xh^n}{\tau},\etah\right)_{\Gam_h^{n+\frac12}}
 &=(\Rop_h^{n+\frac12}\bm u_h^{n+1/2},\etah)_{\Gam_h^{n+\frac12}},
 \label{eq:discrete-kinematic}\\
 r^{n+1}-r^n
 &=\frac12L_h^{n+\frac12}[\Xh^{n+1}-\Xh^n],
 \label{eq:discrete-sav}
\end{align}
\end{subequations}
The right-hand side is zero because gravity is contained in
\(\F_h\) and \(L_h^{n+\frac12}\).

\paragraph{Step 3: Update the mesh and velocity to \(t_{n+1}\).}
After solving
\eqref{eq:fully-discrete-scheme}, compute the harmonic update velocity
\(\widetilde{\bm w}_h^{n+\frac12}
\in W_h[\x^{n+\frac12}]\):
\begin{subequations}\label{eq:discrete-update-flow-map}
\begin{align}
 \widetilde{\bm w}_h^{n+\frac12}
 |_{\Gam_h^{n+\frac12}}
 &=\Rop_h^{n+\frac12}\bm u_h^{n+1/2},\\
 \sum_{\pm}\int_{\Om_{h,\pm}^{n+\frac12}}
 \nabla\widetilde{\bm w}_h^{n+\frac12}:\nabla\bm z_h\dd x&=0
 &&\forall\bm z_h\in\mathring W_h[\x^{n+\frac12}],\\
 \bm x_j^{n+1}
 &=\bm x_j^n+\tau\widetilde{\bm w}_h^{n+\frac12}
 (\bm x_j^{n+\frac12}),\\
 \bm\phi_h^{n+1}(\bm x_j^0)&=\bm x_j^{n+1}.
\end{align}
\end{subequations}
For interface geometry nodes, the trace condition for
\(\widetilde{\bm w}_h^{n+\frac12}\) and the nodal update reproduce
\eqref{eq:discrete-kinematic}; hence the discrete ALE flow map and
\(\Xh^{n+1}\) have the same endpoint interface.
After this recovery, define
\begin{equation}\label{eq:outgoing-transfer}
 \uh^{n+1}=\mathcal T_{t_{n+\frac12}\to t_{n+1}}
 \widehat{\bm u}_h^{\,n+1}.
\end{equation}

\begin{remark}
The trace projection need not give
\(\betah^{n+\frac12}\cdot\bm\nu_h=0\) pointwise on
\(\Gam_h^{n+\frac12}\).  Suppressing the time superscripts and writing
\(c\) for the exactly integrated form of \(c_h^{n+\frac12}\), phasewise
integration by parts gives
\begin{align}
 \begin{aligned}
 c(\bm\beta_h;\bm z_h,\bm v_h)
 ={}&\bigl(\rho_h\{(\bm\beta_h\cdot\nabla)\bm z_h
             -\tfrac12(\Div\bm w_h)\bm z_h\},\bm v_h\bigr)_{\Om_h}\\
 &+\frac12\bigl(\rho_h(\Div\bm u_{\mathrm{adv},h})\bm z_h,
                  \bm v_h\bigr)_{\Om_h}
 -\frac12(\rho_--\rho_+)\int_{\Gam_h}
   (\bm\beta_h\cdot\bm\nu_h)(\bm z_h\cdot\bm v_h)\dd s .
 \end{aligned}
\end{align}
The last two terms vanish when the advecting velocity is divergence-free
and its normal component agrees with the mesh velocity on the interface.
Quadrature affects this identity but not the skew symmetry
\(c_h^{n+\frac12}(\bm\beta_h;\bm z_h,\bm z_h)=0\) used in the energy proof.
\end{remark}
\subsection{Discrete energy decay}

\begin{theorem}[Discrete energy inequality]\label{thm:discrete-energy}
Define
\begin{equation}\label{eq:discrete-modified-energy}
 \M_h^n=\frac12\norm{\sqrt{\rho_h^n}\,\uh^n}_{\Om_h^n}^2
 +(r^n)^2-C_0.
\end{equation}
For every geometrically accepted time step, the solution of
\eqref{eq:fully-discrete-scheme} satisfies
\begin{align}
 \M_h^{n+1}-\M_h^n
 +2\tau\norm{\sqrt{\mu_h^{n+\frac12}}
       \bm D(\bm u_h^{n+1/2})}_{\Om_h^{n+\frac12}}^2\le 0.
 \label{eq:discrete-energy-law}
\end{align}
Consequently, \(\M_h^{n+1}\le\M_h^n\).
\end{theorem}

\begin{proof}
By \eqref{eq:discrete-kinematic} and the positive definiteness of the
surface \(L^2\) product,
\[
 \frac{\Xh^{n+1}-\Xh^n}{\tau}
 =\Rop_h^{n+\frac12}\bm u_h^{n+1/2}
 \quad\text{in }(\mathcal S_h[\x^{n+\frac12}])^d.
\]
Hence \eqref{eq:discrete-sav} gives
\begin{align}
 r^{\,n+1/2}L_h^{n+\frac12}
 [\Rop_h^{n+\frac12}\bm u_h^{n+1/2}]
 &=\frac{r^{\,n+1/2}}{\tau}
   L_h^{n+\frac12}[\Xh^{n+1}-\Xh^n]\\
 &=\frac{2r^{\,n+1/2}}{\tau}(r^{n+1}-r^n)
 =\frac{(r^{n+1})^2-(r^n)^2}{\tau}. \notag
\end{align}
Take \(\vh=\bm u_h^{n+1/2}\) and \(q_h=p_h^{n+1/2}\).
Skew symmetry, incompressibility, and the midpoint identity yield
\begin{align*}
 c_h^{n+\frac12}
 (\betah^{n+\frac12};
 \bm u_h^{n+1/2},\bm u_h^{n+1/2})&=0,\\
 b_h^{n+\frac12}
 (\bm u_h^{n+1/2},p_h^{n+1/2})&=0,\\
 a_h^{n+\frac12}
 (\bm u_h^{n+1/2},\bm u_h^{n+1/2})
 &=2\norm{\sqrt{\mu_h^{n+\frac12}}
 \bm D(\bm u_h^{n+1/2})}_{\Om_h^{n+\frac12}}^2,\\
 \left(\rho_h^{n+\frac12}
       \frac{\widehat{\bm u}_h^{\,n+1}
                   -\widehat{\bm u}_h^{\,n}}{\tau},
       \bm u_h^{n+1/2}\right)_{\Om_h^{n+\frac12}}
 &=\frac{\norm{\sqrt{\rho_h^{n+\frac12}}\,
                 \widehat{\bm u}_h^{\,n+1}}_{\Om_h^{n+\frac12}}^2
          -\norm{\sqrt{\rho_h^{n+\frac12}}\,
                 \widehat{\bm u}_h^{\,n}}_{\Om_h^{n+\frac12}}^2}{2\tau}.
\end{align*}
Thus \eqref{eq:discrete-momentum} gives the carrier-mesh balance
\begin{align}
 &\frac12\left(
 \norm{\sqrt{\rho_h^{n+\frac12}}\,
       \widehat{\bm u}_h^{\,n+1}}_{\Om_h^{n+\frac12}}^2
 -\norm{\sqrt{\rho_h^{n+\frac12}}\,
       \widehat{\bm u}_h^{\,n}}_{\Om_h^{n+\frac12}}^2\right)
 +(r^{n+1})^2-(r^n)^2\\
 &\qquad+2\tau\norm{\sqrt{\mu_h^{n+\frac12}}
 \bm D(\bm u_h^{n+1/2})}_{\Om_h^{n+\frac12}}^2
 =0. \notag
\end{align}
Applying \Cref{lem:ale-transfer-stability} to the incoming and outgoing
transfers gives
\begin{align*}
 \norm{\sqrt{\rho_h^{n+\frac12}}\,
       \widehat{\bm u}_h^{\,n}}_{\Om_h^{n+\frac12}}^2
 &\le\norm{\sqrt{\rho_h^n}\,\uh^n}_{\Om_h^n}^2,\\
 \norm{\sqrt{\rho_h^{n+1}}\,\uh^{n+1}}_{\Om_h^{n+1}}^2
 &\le\norm{\sqrt{\rho_h^{n+\frac12}}\,
       \widehat{\bm u}_h^{\,n+1}}_{\Om_h^{n+\frac12}}^2.
\end{align*}
Combining these with the carrier-mesh balance gives
\begin{align*}
 \M_h^{n+1}-\M_h^n
 +2\tau\norm{\sqrt{\mu_h^{n+\frac12}}
 \bm D(\bm u_h^{n+1/2})}_{\Om_h^{n+\frac12}}^2
 &\le0.
\end{align*}
The proof is complete.
\end{proof}

\section{Numerical results in two dimensions}\label{sec:numerics}

All computations use the fully discrete scheme with \(C_0=1\),
\(r^0=r_h(\Xh^0)\), and no pressure penalty.  The capillary test uses
degree-\(k\) isoparametric elements, \(k=1,2,3\), with eighth-order
quadrature; the rising-bubble tests use \(k=1\) and fifth-order quadrature.

The physical discrete energy is
\begin{equation}\label{eq:discrete-physical-energy}
 \mathcal P_h^n
 =\frac12\norm{\sqrt{\rho_h^n}\,\uh^n}_{\Om_h^n}^2
 +\F_h(\Xh^n).
\end{equation}
Since the scalar update need not preserve \(r^n=r_h(\Xh^n)\), define
\begin{equation}\label{eq:sav-gap}
 d_h^n=(r^n)^2-C_0-\F_h(\Xh^n),
 \qquad
 \M_h^n-\mathcal P_h^n=d_h^n.
\end{equation}
Thus \Cref{thm:discrete-energy} controls \(\M_h^n\), whereas any monotonicity
of \(\mathcal P_h^n\) reported below is an observed numerical property.

The method is implemented in NGSolve, with the coupled systems solved
monolithically.  No remeshing is used in the refinement studies.  Reported
balance residuals refer to the carrier-mesh identity in the proof of
\Cref{thm:discrete-energy}.

For self-convergence, velocities are pulled back to the initial domain.  At
adjacent resolutions \(q>q'\), where \(q=\tau\) or the measured bulk mesh
size \(h\), we compute
\begin{align*}
 d_{u;q,q'}
 &=\norm{(\bm u_q(T)\circ\bm\phi_q(T))
          -(\bm u_{q'}(T)\circ\bm\phi_{q'}(T))}_{L^2(\Om^0)},\\
 d_{\Gamma;q,q'}
 &=\frac{\norm{\bm X_q(T)-\bm X_{q'}(T)}_{L^2(\Gam^0)}}
         {\norm{\bm X^0}_{L^2(\Gam^0)}}.
\end{align*}
For three resolutions \(q_0>q_1>q_2\), the order \(p\) is determined from
\[
 \frac{d_{q_0,q_1}}{d_{q_1,q_2}}
 =\frac{q_0^p-q_1^p}{q_1^p-q_2^p}.
\]
All differences are evaluated at common quadrature points and are
self-differences rather than errors against an exact solution.

\subsection{Force-free capillary relaxation}\label{sec:capillary}

The container is \([-1.25,1.25]\times[-1,1]\).  The initial interface is a
stationary ellipse with semiaxes \(0.62\) and \(0.38\).  There is no gravity,
the outer boundary is no slip, and
\[
 (\rho_+,\rho_-)=(1,0.8),\qquad
 (\mu_+,\mu_-)=(0.08,0.04),\qquad \gamma=1.
\]
For the degree comparison, \(h_{\max}=0.18\), the initial interface has 32
segments, and \(\tau=0.0025\).  The energy runs end at \(T=0.15\), while the
\(k=2\) solution in \Cref{fig:capillary-relaxation-flow} is continued to
\(T=0.2\).

\begin{figure}[htbp]
\centering
\includegraphics[width=\textwidth]{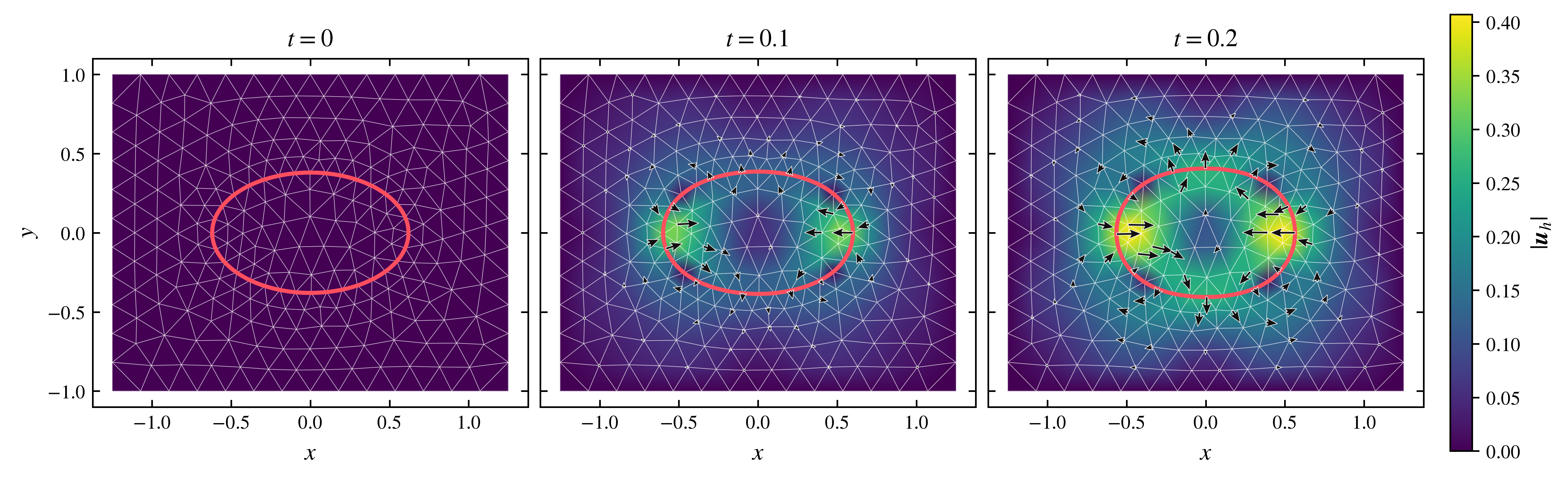}
\caption{Capillary relaxation for \(k=2\).}
\label{fig:capillary-relaxation-flow}
\end{figure}

For temporal refinement, \(T=0.02\), \(h_{\max}=0.30\), the initial interface
has 20 segments, and \(\tau=0.005\,2^{-j}\), \(j=0,\ldots,4\).  The velocity
and interface differences have order two for \(k=1,2,3\); see
\Cref{fig:capillary-temporal-velocity,fig:capillary-temporal-interface}.
For spatial refinement, \(T=0.01\) and \(\tau=10^{-4}\); the observed orders
are two, three, and four, respectively, as shown in
\Cref{fig:capillary-spatial-convergence,fig:capillary-spatial-interface}.
The modified energy decreases for each degree, with a maximum balance
residual below \(2.4\times10^{-14}\).  The physical energy also decreases in
this test, although this is not implied by \Cref{thm:discrete-energy}.

\begingroup
\makeatletter
\setlength{\@fptop}{0pt plus 1fil}
\setlength{\@fpsep}{14pt}
\setlength{\@fpbot}{0pt plus 1fil}
\makeatother
\captionsetup[subfigure]{skip=3pt}
\begin{figure}[htbp]
\centering
\begin{subfigure}[t]{0.475\textwidth}
  \centering
  \includegraphics[width=\linewidth]{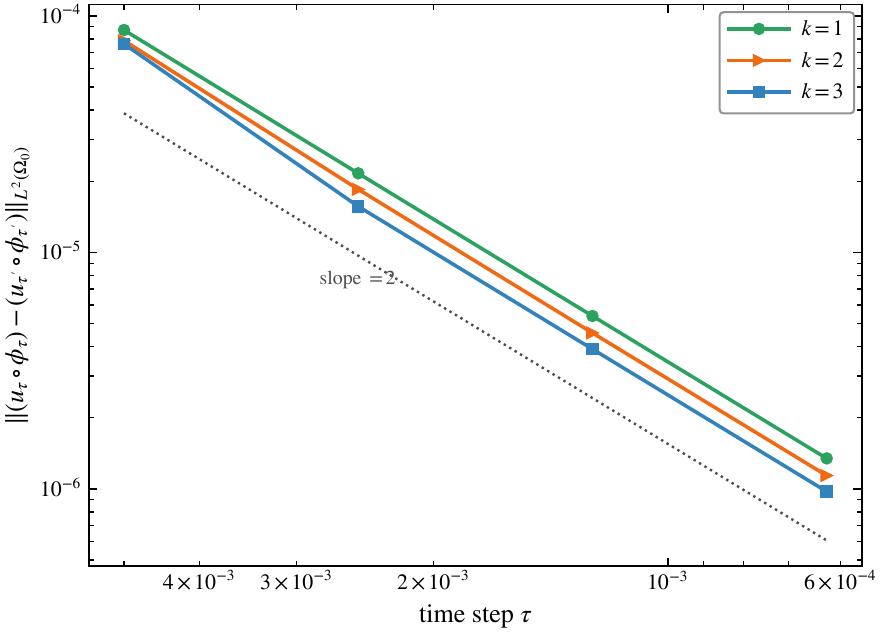}
  \caption{Temporal: velocity.}
  \label{fig:capillary-temporal-velocity}
\end{subfigure}\hfill
\begin{subfigure}[t]{0.475\textwidth}
  \centering
  \includegraphics[width=\linewidth]{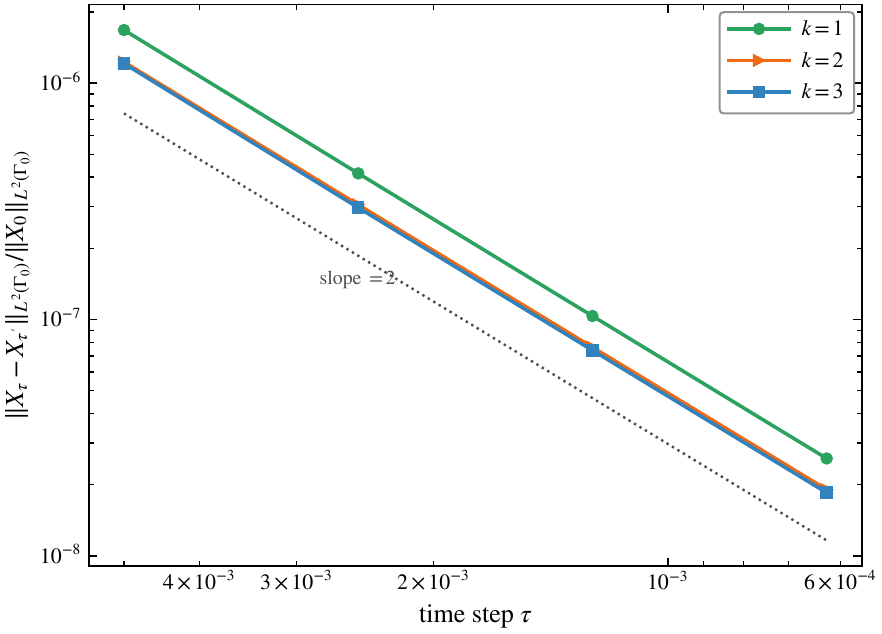}
  \caption{Temporal: interface.}
  \label{fig:capillary-temporal-interface}
\end{subfigure}
\par\vspace{0.4\baselineskip}
\begin{subfigure}[t]{0.475\textwidth}
  \centering
  \includegraphics[width=\linewidth,trim={0 0 431.16bp 0},clip]{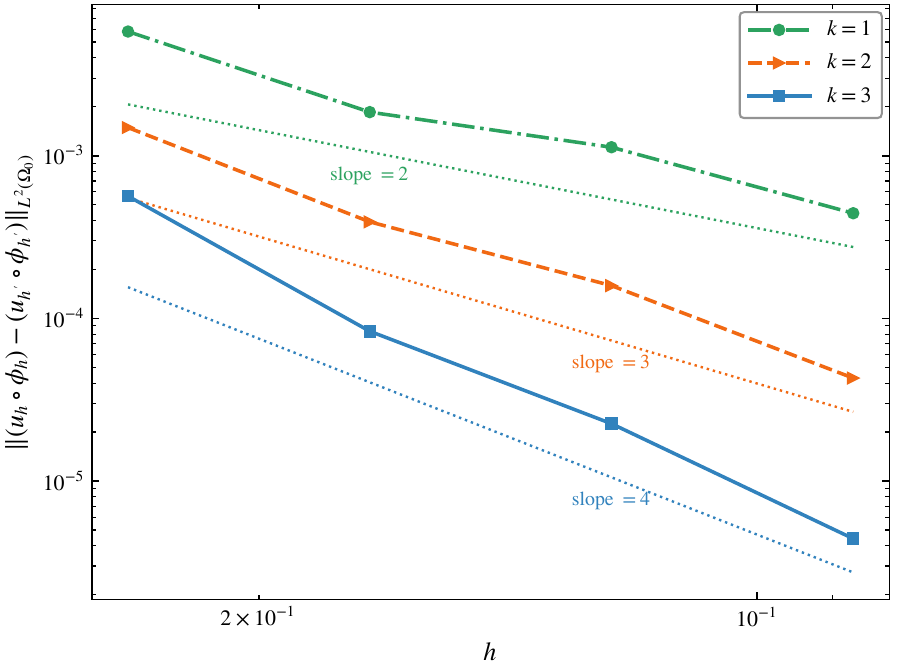}
  \caption{Spatial: velocity.}
  \label{fig:capillary-spatial-convergence}
\end{subfigure}\hfill
\begin{subfigure}[t]{0.475\textwidth}
  \centering
  \includegraphics[width=\linewidth,trim={431.16bp 0 0 0},clip]{spatial_errors_ratio14.pdf}
  \caption{Spatial: interface flow map.}
  \label{fig:capillary-spatial-interface}
\end{subfigure}
\caption{Temporal and spatial convergence.}
\label{fig:capillary-convergence}
\end{figure}
\begin{figure}[htbp]
\centering
\begin{subfigure}[t]{0.475\textwidth}
  \centering
  \includegraphics[width=\linewidth]{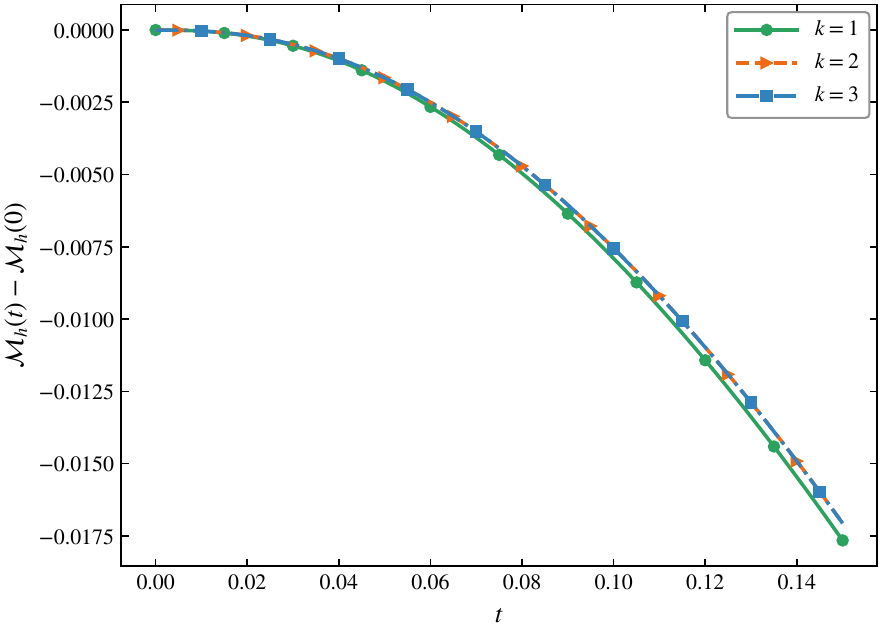}
  \caption{Modified energy \(\M_h(t)-\M_h(0)\).}
  \label{fig:capillary-modified-energy}
\end{subfigure}\hfill
\begin{subfigure}[t]{0.475\textwidth}
  \centering
  \includegraphics[width=\linewidth]{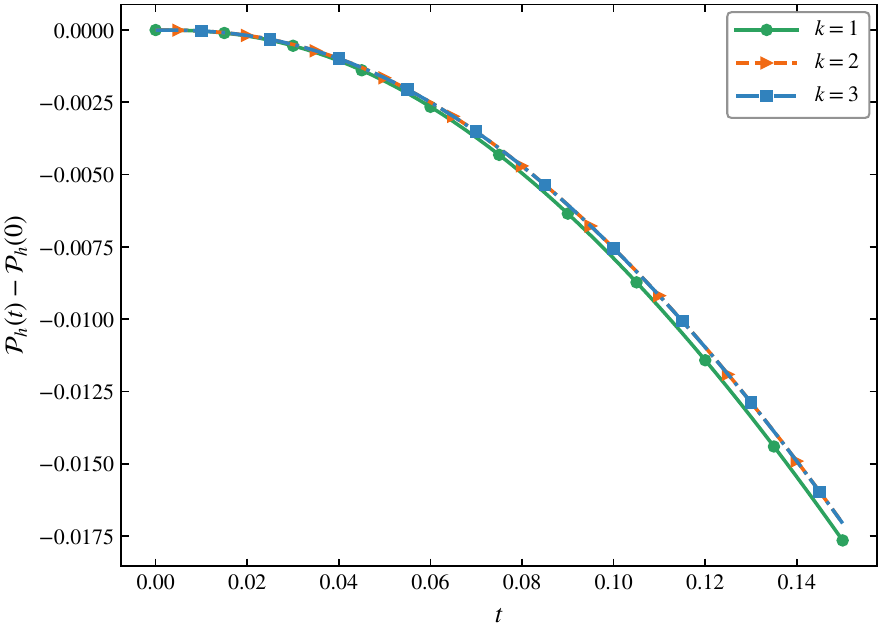}
  \caption{Physical energy \(\mathcal P_h(t)-\mathcal P_h(0)\).}
  \label{fig:capillary-physical-energy}
\end{subfigure}
\caption{Energy decay over time.}
\label{fig:capillary-energy-decay}
\end{figure}

\FloatBarrier  
\endgroup
\subsection{Benchmark problem 1}\label{sec:bubble}

We use the rising-bubble benchmarks of \citet{HysingEtAl2009} on
\(\Om=(0,1)\times(0,2)\), with an initial bubble of radius \(0.25\) centred at
\((0.5,0.5)\).  The horizontal walls are no slip, the vertical walls are free
slip, and \(g=0.98\).  The two parameter sets are listed in
\Cref{tab:bubble-parameters}; we first consider BP-1.

\begin{table}[htbp]
\centering
\caption{Rising-bubble parameters.}
\label{tab:bubble-parameters}
\begin{tabular}{cccccc}
\toprule
problem & \(\rho_+\) & \(\rho_-\) & \(\mu_+\) & \(\mu_-\) & \(\gamma\)\\
\midrule
BP-1 & 1000 & 100 & 10 & 1 & 24.5\\
BP-2 & 1000 & 1 & 10 & 0.1 & 1.96\\
\bottomrule
\end{tabular}
\end{table}

For the bubble area \(A_h(t)=\abs{\Om_{h,-}(t)}\), the benchmark quantities are
\begin{equation}\label{eq:bp1-benchmark-quantities}
 \operatorname{Cir}_h(t)
 =\frac{2\sqrt{\pi A_h(t)}}{\abs{\Gam_h(t)}},\qquad
 y_c(t)=\frac{1}{A_h(t)}\int_{\Om_{h,-}(t)}x_2\,\mathrm dx,
 \qquad
 V_c(t)=\frac{1}{A_h(t)}\int_{\Om_{h,-}(t)}(\uh)_2\,\mathrm dx.
\end{equation}
We also report the physical energy \(\mathcal P_h\) in
\eqref{eq:discrete-physical-energy}.

The initial fitted mesh has 2779 vertices, 5346 triangles, 53 interface edges,
and \(h_{\max}=0.0285\).  We take \(\tau=1/128,1/256,1/512\), \(T=2.5\), and
regenerate poor-quality meshes as in \citet{DuanLiYang2022}.

Temporal self-convergence is measured on \(0\le t\le T_0=0.5\), before the
first remeshing.  For \(q\in\{\operatorname{Cir}_h,y_c,V_c\}\), let \(q_\tau\)
be its discrete history and set \(t_i=i\tau/2\), \(i=1,\ldots,M\), with
\(M=2T_0/\tau\).  The relative differences are
\begin{align}
 \norm{e_q(\tau,\tau/2)}_{\ell^2}
 &:=\left(
 \frac{\displaystyle\sum_{i=1}^{M}
 \abs{(\mathcal I_{\tau\to\tau/2}q_\tau)(t_i)-q_{\tau/2}(t_i)}^2}
 {\displaystyle\sum_{i=1}^{M}\abs{q_{\tau/2}(t_i)}^2}
 \right)^{1/2},
 \label{eq:bp1-temporal-error-l2}\\
 \norm{e_q(\tau,\tau/2)}_{\ell^\infty}
 &:=
 \frac{\displaystyle\max_{1\le i\le M}
 \abs{(\mathcal I_{\tau\to\tau/2}q_\tau)(t_i)-q_{\tau/2}(t_i)}}
 {\displaystyle\max_{1\le i\le M}\abs{q_{\tau/2}(t_i)}}.
 \label{eq:bp1-temporal-error-linf}
\end{align}
Here \(\mathcal I_{\tau\to\tau/2}\) denotes piecewise linear interpolation;
rates are computed from the base-two logarithm of successive differences.

\begin{table}[htbp]
\centering
\caption{Relative errors and convergence rates in time for BP-1.}
\label{tab:bp1-temporal-convergence}
\small
\begin{tabular}{llcccc}
\toprule
quantity & \(\tau\)
& \(\norm{e_q(\tau,\tau/2)}_{\ell^2}\) & rate
& \(\norm{e_q(\tau,\tau/2)}_{\ell^\infty}\) & rate\\
\midrule
Circularity
& \(1/128\) & \(1.78\times10^{-7}\) & -- & \(7.41\times10^{-7}\) & --\\
& \(1/256\) & \(4.44\times10^{-8}\) & \(2.00\) & \(1.88\times10^{-7}\) & \(1.98\)\\
\midrule
Centre of mass
& \(1/128\) & \(3.71\times10^{-6}\) & -- & \(7.30\times10^{-6}\) & --\\
& \(1/256\) & \(9.26\times10^{-7}\) & \(2.00\) & \(1.84\times10^{-6}\) & \(1.99\)\\
\midrule
Rise velocity
& \(1/128\) & \(2.41\times10^{-5}\) & -- & \(4.52\times10^{-5}\) & --\\
& \(1/256\) & \(6.06\times10^{-6}\) & \(1.99\) & \(1.29\times10^{-5}\) & \(1.81\)\\
\bottomrule
\end{tabular}
\end{table}

\Cref{tab:bp1-temporal-convergence} shows second-order temporal
self-convergence on \([0,T_0]\).  The full histories in
\Cref{fig:bp1-time-step-comparison} include remeshing and are not used for the
rate estimate.  Pressure and velocity for the finest run are shown in
\Cref{fig:bp1-pressure,fig:bp1-velocity}.  Both the modified and physical
energies decrease in these computations.

\begin{figure}[htbp]
\centering
\includegraphics[width=0.8\textwidth]{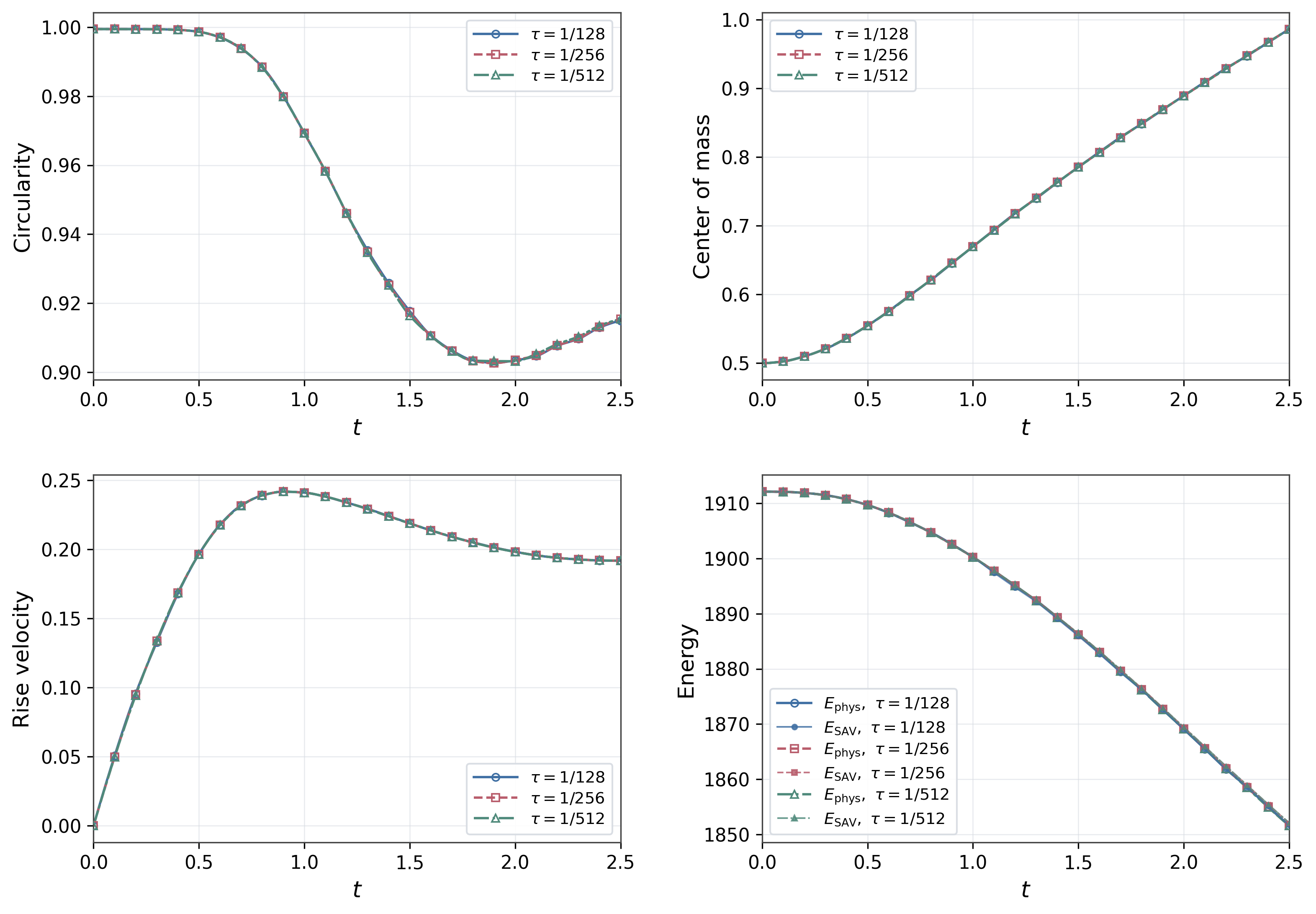}
\caption{BP-1 benchmark quantities for different time-step sizes.}
\label{fig:bp1-time-step-comparison}
\end{figure}

\begin{figure}[htbp]
\centering
\includegraphics[width=0.8\textwidth]{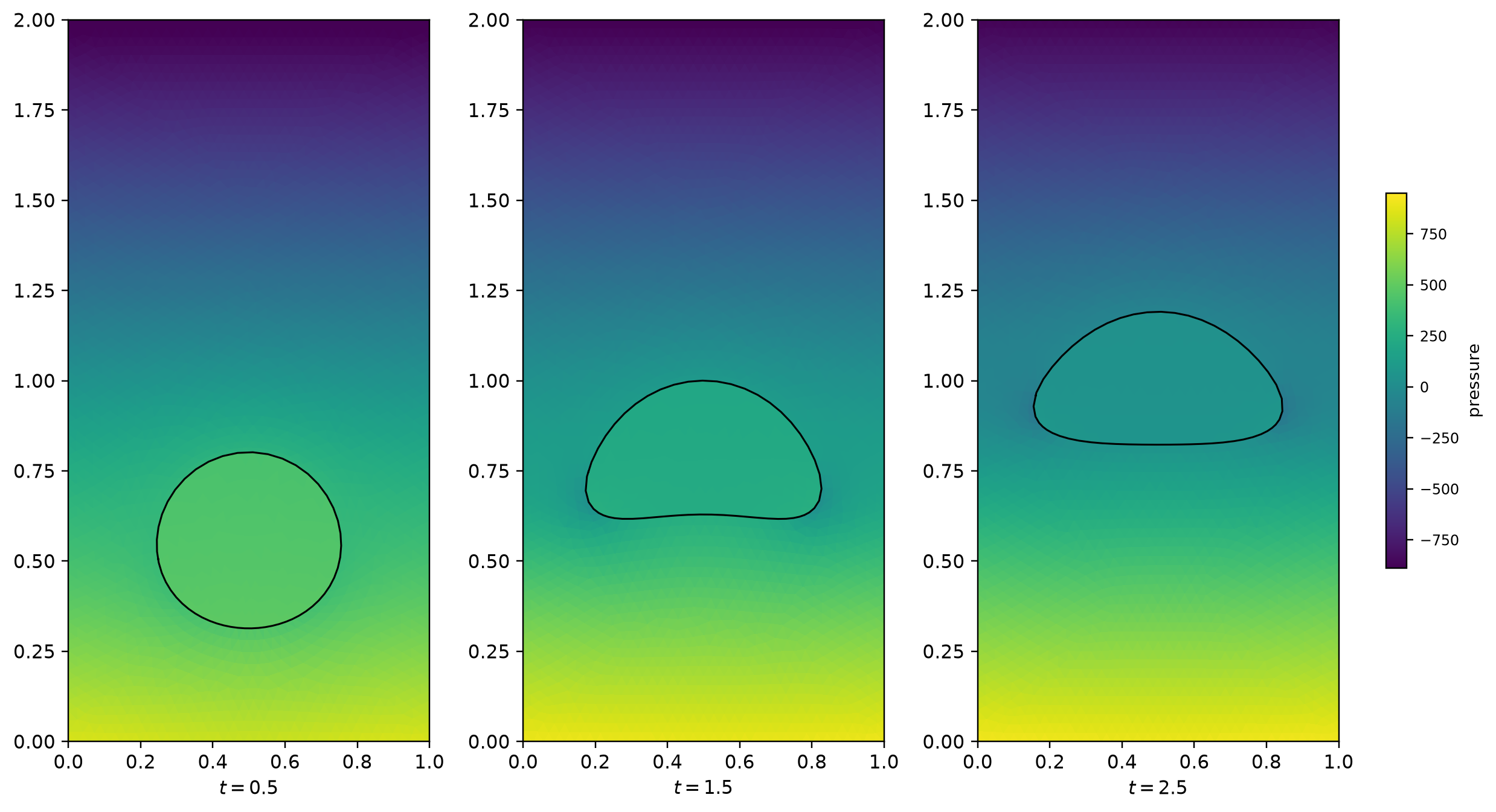}
\caption{Pressure for BP-1 at \(t=0.5,1.5,2.5\), computed with
\(\tau=1/512\).}
\label{fig:bp1-pressure}
\end{figure}

\begin{figure}[htbp]
\centering
\includegraphics[width=0.8\textwidth]{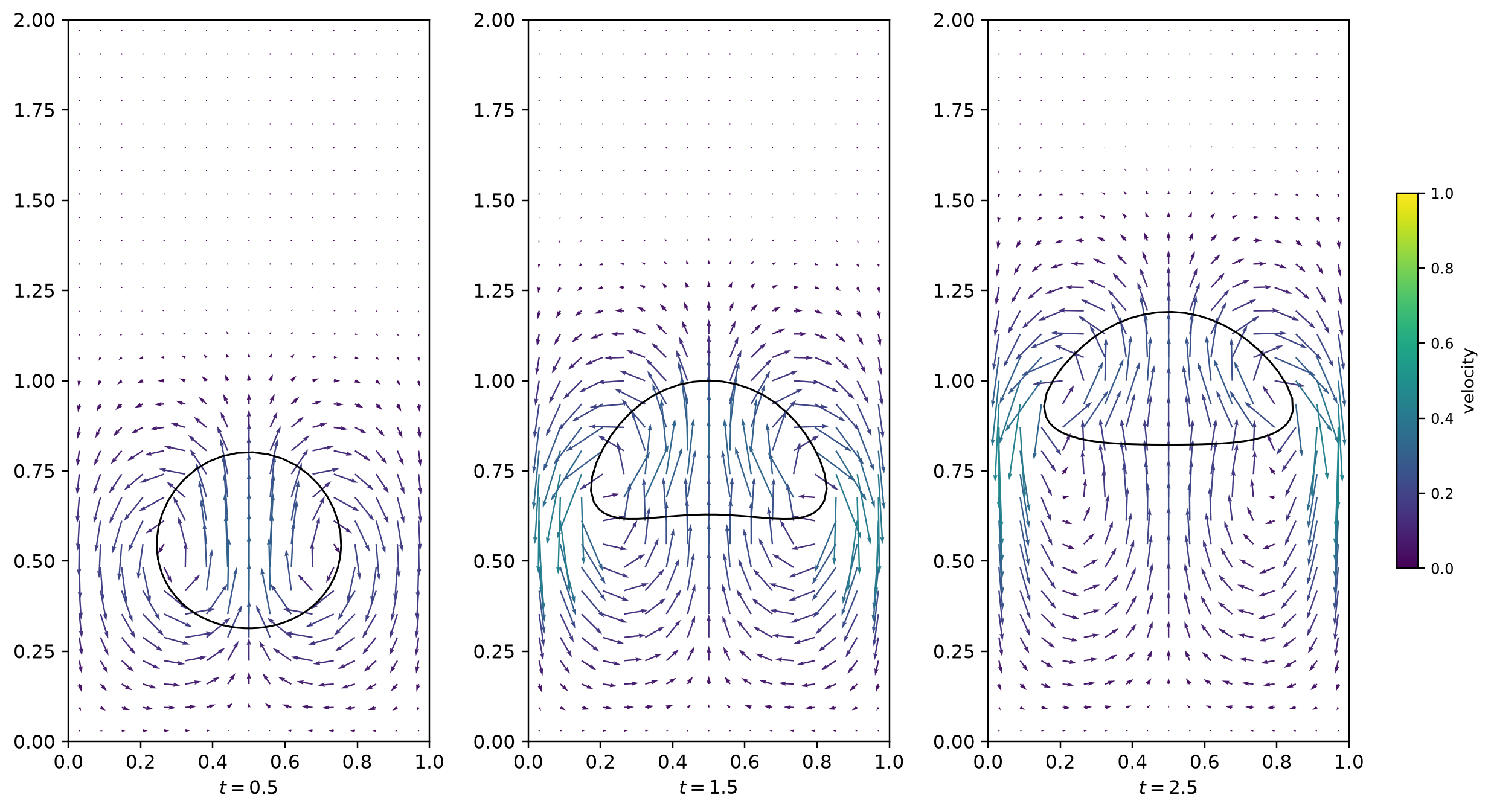}
\caption{Velocity for BP-1 at \(t=0.5,1.5,2.5\), computed with
\(\tau=1/512\).}
\label{fig:bp1-velocity}
\end{figure}

\FloatBarrier
\subsection{Benchmark problem 2}\label{sec:bubble-case2}

BP-2 uses the same geometry and boundary conditions, with the parameters in
\Cref{tab:bubble-parameters}.  The initial mesh has \(h_{\max}=0.03\) and 80
interface edges; we use \(\tau=1/256\), \(T=2.5\), and the BP-1 remeshing
procedure.  The maximum accepted-step balance residual is
\(3.829\times10^{-11}\).  The benchmark histories are shown in
\Cref{fig:bp2-benchmark-quantities}, and the pressure and velocity in
\Cref{fig:bp2-pressure,fig:bp2-velocity}.  Small late-time variations occur
in the circularity and rise velocity.  Since only one resolution is used, no
convergence order is reported.

\begin{figure}[htbp]
\centering
\includegraphics[width=0.8\textwidth]{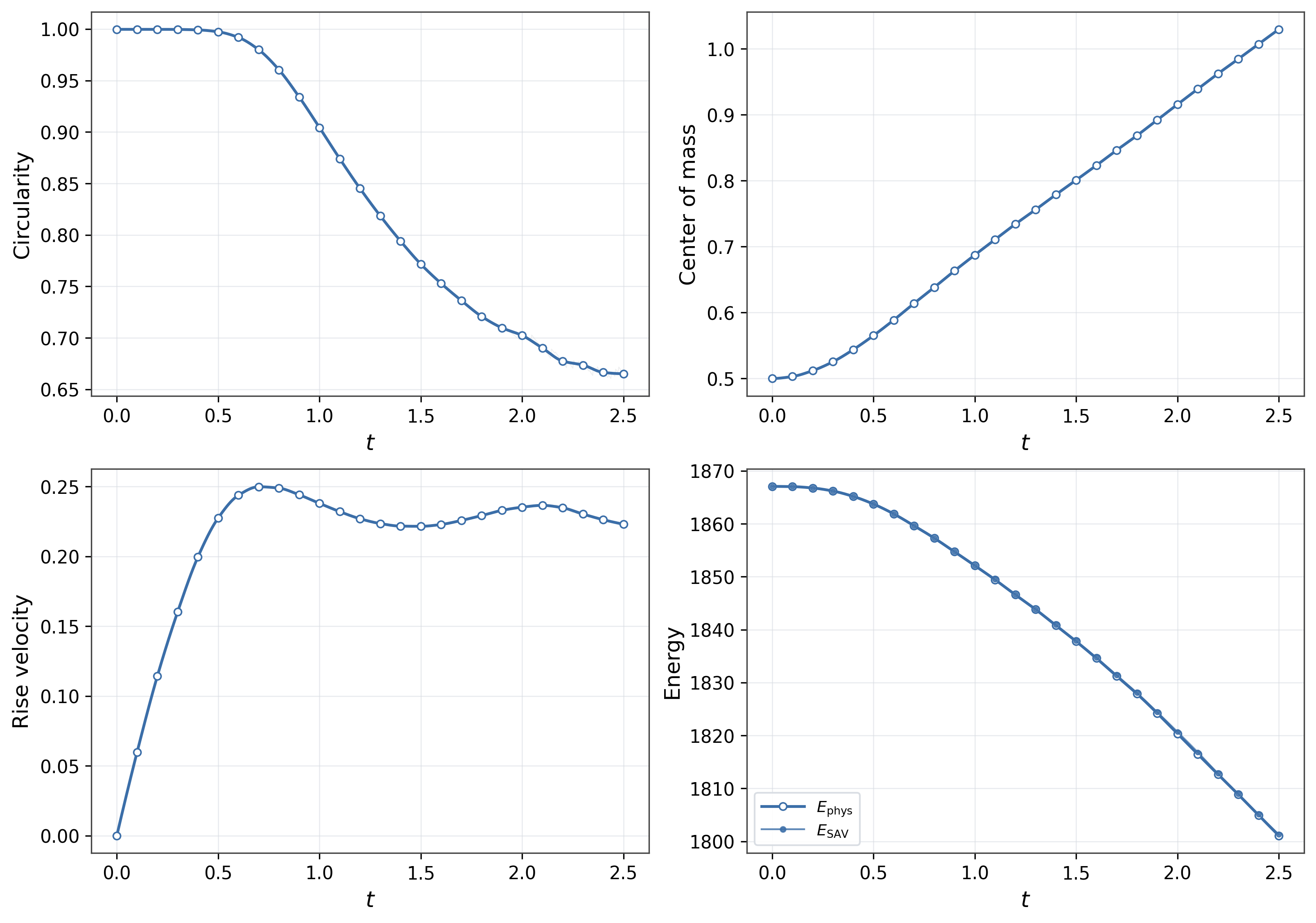}
\caption{BP-2 benchmark quantities.}
\label{fig:bp2-benchmark-quantities}
\end{figure}

\begin{figure}[htbp]
\centering
\includegraphics[width=0.8\textwidth]{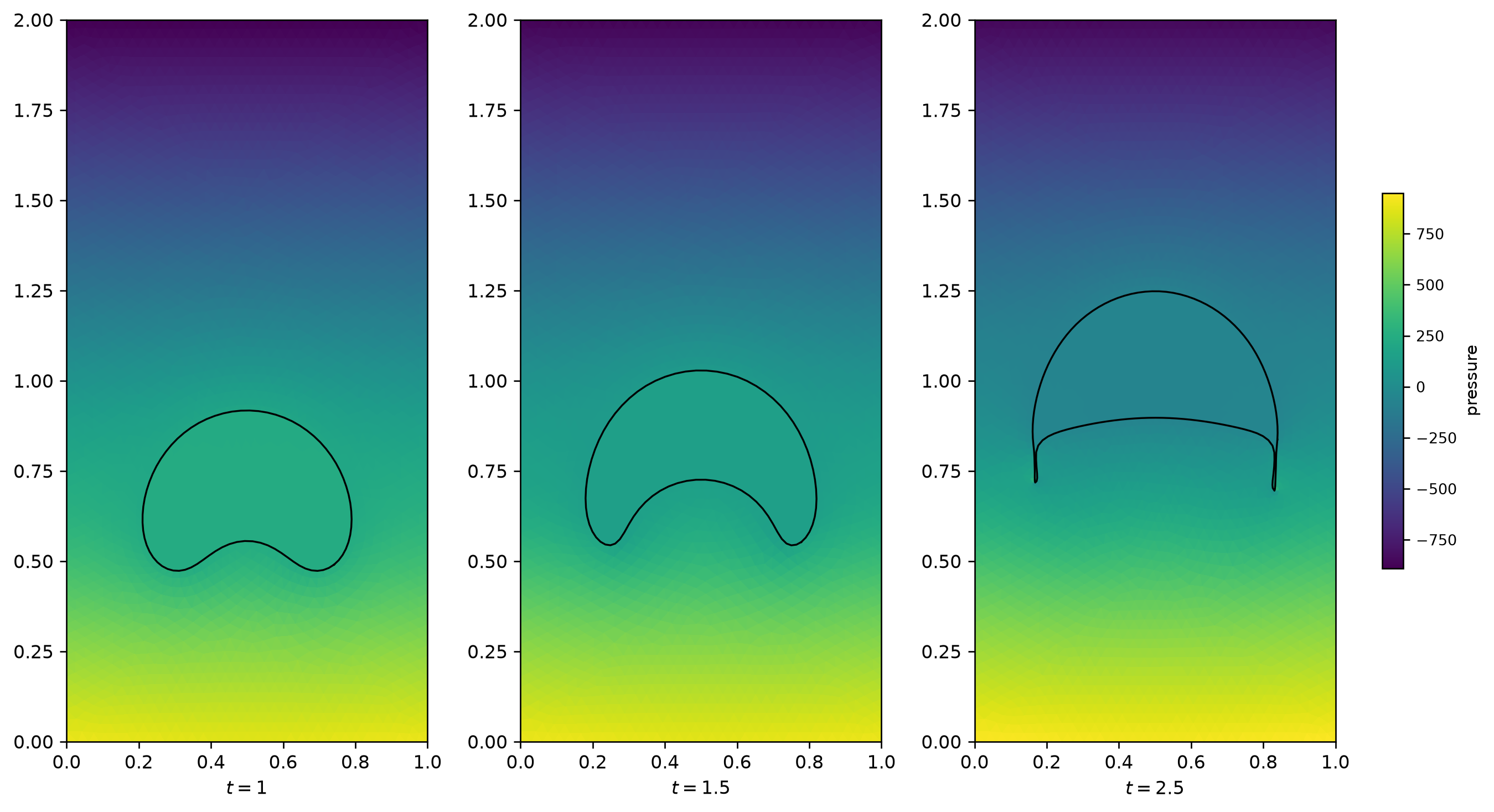}
\caption{Pressure for BP-2 at \(t=1,1.5,2.5\), computed with
\(\tau=1/256\).}
\label{fig:bp2-pressure}
\end{figure}

\begin{figure}[htbp]
\centering
\includegraphics[width=0.8\textwidth]{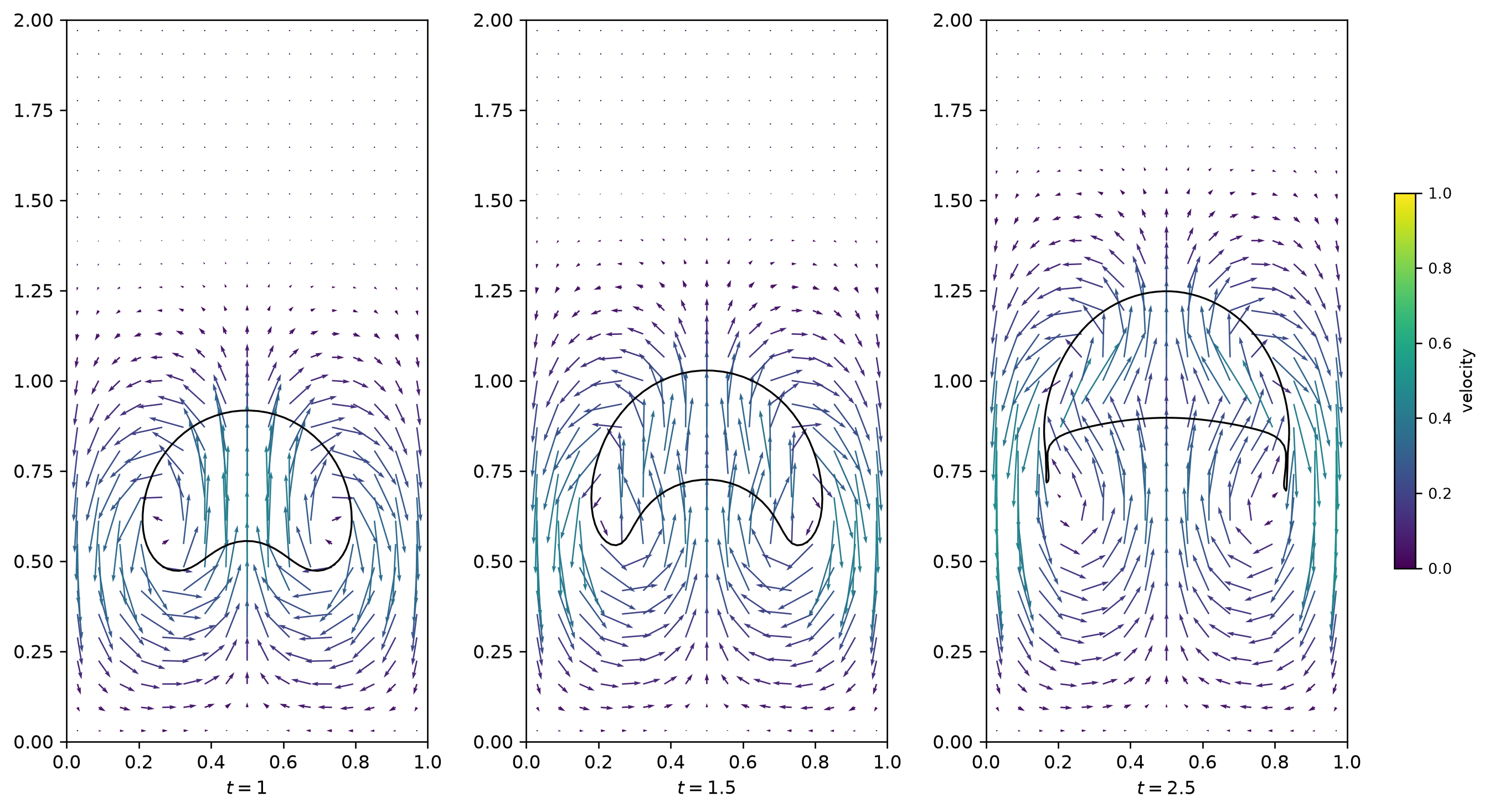}
\caption{Velocity for BP-2 at \(t=1,1.5,2.5\), computed with
\(\tau=1/256\).}
\label{fig:bp2-velocity}
\end{figure}

\bibliographystyle{abbrvnat}
\bibliography{references}

@article{AgneseNurnberg2020,
  author  = {Agnese, Marco and N{\"u}rnberg, Robert},
  title   = {Fitted front tracking methods for two-phase incompressible {N}avier--{S}tokes flow: {E}ulerian and {ALE} finite element discretizations},
  journal = {International Journal of Numerical Analysis and Modeling},
  volume  = {17},
  number  = {5},
  pages   = {613--642},
  year    = {2020},
  eprint  = {1910.14327},
  archiveprefix = {arXiv}
}

@article{BadiaCodina2006,
  author  = {Badia, Santiago and Codina, Ramon},
  title   = {Analysis of a stabilized finite element approximation of the transient convection-diffusion equation using an {ALE} framework},
  journal = {SIAM Journal on Numerical Analysis},
  volume  = {44},
  number  = {5},
  pages   = {2159--2197},
  year    = {2006},
  doi     = {10.1137/050643532}
}

@article{Bansch2001,
  author  = {B{\"a}nsch, Eberhard},
  title   = {Finite element discretization of the {N}avier--{S}tokes equations with a free capillary surface},
  journal = {Numerische Mathematik},
  volume  = {88},
  number  = {2},
  pages   = {203--235},
  year    = {2001},
  doi     = {10.1007/PL00005443}
}

@article{BarrettGarckeNurnberg2015,
  author  = {Barrett, John W. and Garcke, Harald and N{\"u}rnberg, Robert},
  title   = {A stable parametric finite element discretization of two-phase {N}avier--{S}tokes flow},
  journal = {Journal of Scientific Computing},
  volume  = {63},
  number  = {1},
  pages   = {78--117},
  year    = {2015},
  doi     = {10.1007/s10915-014-9885-2}
}

@incollection{BarrettGarckeNurnberg2020,
  author    = {Barrett, John W. and Garcke, Harald and N{\"u}rnberg, Robert},
  title     = {Parametric finite element approximations of curvature-driven interface evolutions},
  booktitle = {Geometric Partial Differential Equations---Part I},
  series    = {Handbook of Numerical Analysis},
  volume    = {21},
  pages     = {275--423},
  publisher = {Elsevier},
  year      = {2020},
  doi       = {10.1016/bs.hna.2019.05.002}
}

@article{BoffiGastaldi2004,
  author  = {Boffi, Daniele and Gastaldi, Lucia},
  title   = {Stability and geometric conservation laws for {ALE} formulations},
  journal = {Computer Methods in Applied Mechanics and Engineering},
  volume  = {193},
  number  = {42--44},
  pages   = {4717--4739},
  year    = {2004},
  doi     = {10.1016/j.cma.2004.02.020}
}

@article{BonitoKyzaNochetto2013,
  author  = {Bonito, Andrea and Kyza, Irene and Nochetto, Ricardo H.},
  title   = {Time-discrete higher-order {ALE} formulations: Stability},
  journal = {SIAM Journal on Numerical Analysis},
  volume  = {51},
  number  = {1},
  pages   = {577--604},
  year    = {2013},
  doi     = {10.1137/120862715}
}

@article{DuanLiYang2022,
  author  = {Duan, Beiping and Li, Buyang and Yang, Zongze},
  title   = {An energy diminishing arbitrary {L}agrangian--{E}ulerian finite element method for two-phase {N}avier--{S}tokes flow},
  journal = {Journal of Computational Physics},
  volume  = {461},
  pages   = {111215},
  year    = {2022},
  doi     = {10.1016/j.jcp.2022.111215}
}

@article{Dziuk1991,
  author  = {Dziuk, Gerhard},
  title   = {An algorithm for evolutionary surfaces},
  journal = {Numerische Mathematik},
  volume  = {58},
  pages   = {603--611},
  year    = {1991},
  doi     = {10.1007/BF01385643}
}

@article{Fu2020,
  author  = {Fu, Guosheng},
  title   = {Arbitrary {L}agrangian--{E}ulerian hybridizable discontinuous {G}alerkin methods for incompressible flow with moving boundaries and interfaces},
  journal = {Computer Methods in Applied Mechanics and Engineering},
  volume  = {367},
  pages   = {113158},
  year    = {2020},
  doi     = {10.1016/j.cma.2020.113158}
}

@article{GarckeNurnbergZhao2023,
  author  = {Garcke, Harald and N{\"u}rnberg, Robert and Zhao, Quan},
  title   = {Structure-preserving discretizations of two-phase {N}avier--{S}tokes flow using fitted and unfitted approaches},
  journal = {Journal of Computational Physics},
  volume  = {489},
  pages   = {112276},
  year    = {2023},
  doi     = {10.1016/j.jcp.2023.112276}
}

@article{GarckeTrautweinZhang2026,
  author  = {Garcke, Harald and Trautwein, Dennis and Zhang, Ganghui},
  title   = {Structure-preserving parametric finite element methods for two-phase {S}tokes flow based on {L}agrange multiplier approaches},
  journal = {Journal of Computational Physics},
  volume  = {560},
  pages   = {114922},
  year    = {2026},
  doi     = {10.1016/j.jcp.2026.114922},
  eprint  = {2508.12326},
  archiveprefix = {arXiv}
}

@article{HysingEtAl2009,
  author  = {Hysing, Shu-Ren and Turek, Stefan and Kuzmin, Dmitri and Parolini, Nicola and Burman, Erik and Ganesan, Sashikumaar and Tobiska, Lutz},
  title   = {Quantitative benchmark computations of two-dimensional bubble dynamics},
  journal = {International Journal for Numerical Methods in Fluids},
  volume  = {60},
  number  = {11},
  pages   = {1259--1288},
  year    = {2009},
  doi     = {10.1002/fld.1934}
}

@article{HuLeiLiTang2026,
  author  = {Hu, Jiashun and Lei, Nuo and Li, Buyang and Tang, Rong},
  title   = {Energy dissipating {ALE--MDR} method for {N}avier--{S}tokes free boundary problems with moving contact line},
  journal = {SIAM Journal on Scientific Computing},
  volume  = {48},
  number  = {3},
  pages   = {A1284--A1311},
  year    = {2026},
  doi     = {10.1137/25M1784958}
}

@misc{GarckeHuLei2026,
  author       = {Garcke, Harald and Hu, Jiashun and Lei, Nuo},
  title        = {A structure-preserving {ALE--BGN--MDR} method for {N}avier--{S}tokes free boundary problems with moving contact lines and gravity},
  year         = {2026},
  note         = {arXiv:2608.03115},
  doi          = {10.48550/arXiv.2608.03115}
}

@article{JiangSuZhang2024JCP,
  author  = {Jiang, Wei and Su, Chunmei and Zhang, Ganghui},
  title   = {A second-order in time, {BGN}-based parametric finite element method for geometric flows of curves},
  journal = {Journal of Computational Physics},
  volume  = {514},
  pages   = {113220},
  year    = {2024},
  doi     = {10.1016/j.jcp.2024.113220}
}

@article{LiMaQiu2026,
  author  = {Li, Buyang and Ma, Shu and Qiu, Weifeng},
  title   = {Optimal convergence of the arbitrary {L}agrangian--{E}ulerian interface tracking method for two-phase {N}avier--{S}tokes flow without surface tension},
  journal = {IMA Journal of Numerical Analysis},
  volume  = {46},
  number  = {1},
  pages   = {51--89},
  year    = {2026},
  doi     = {10.1093/imanum/draf003}
}

@article{ShenXuYang2018,
  author  = {Shen, Jie and Xu, Jie and Yang, Jiang},
  title   = {The scalar auxiliary variable ({SAV}) approach for gradient flows},
  journal = {Journal of Computational Physics},
  volume  = {353},
  pages   = {407--416},
  year    = {2018},
  doi     = {10.1016/j.jcp.2017.10.021}
}

@article{JiangZhangZhao2022,
  author  = {Jiang, Maosheng and Zhang, Zengyan and Zhao, Jia},
  title   = {Improving the accuracy and consistency of the scalar auxiliary variable ({SAV}) method with relaxation},
  journal = {Journal of Computational Physics},
  volume  = {456},
  pages   = {110954},
  year    = {2022},
  doi     = {10.1016/j.jcp.2022.110954}
}

@article{Lenoir-1986,
  author    = {Lenoir, Marc},
  title     = {Optimal Isoparametric Finite Elements and Error Estimates for Domains Involving Curved Boundaries},
  journal   = {SIAM Journal on Numerical Analysis},
  volume    = {23},
  number    = {3},
  pages     = {562--580},
  year      = {1986},
  doi       = {10.1137/0723036}
}

@book{GrossReusken2011,
  author    = {Gross, Sven and Reusken, Arnold},
  title     = {Numerical Methods for Two-phase Incompressible Flows},
  series    = {Springer Series in Computational Mathematics},
  volume    = {40},
  publisher = {Springer},
  address   = {Berlin},
  year      = {2011},
  doi       = {10.1007/978-3-642-19686-7}
}

@article{HughesLiuZimmermann1981,
  author  = {Hughes, Thomas J. R. and Liu, Wing Kam and Zimmermann, Thomas K.},
  title   = {Lagrangian--Eulerian Finite Element Formulation for Incompressible Viscous Flows},
  journal = {Computer Methods in Applied Mechanics and Engineering},
  volume  = {29},
  number  = {3},
  pages   = {329--349},
  year    = {1981},
  doi     = {10.1016/0045-7825(81)90049-9}
}

@article{FormaggiaNobile2004,
  author  = {Formaggia, Luca and Nobile, Fabio},
  title   = {Stability Analysis of Second-order Time Accurate Schemes for {ALE--FEM}},
  journal = {Computer Methods in Applied Mechanics and Engineering},
  volume  = {193},
  number  = {39--41},
  pages   = {4097--4116},
  year    = {2004},
  doi     = {10.1016/j.cma.2003.09.028}
}

@article{Edelmann2022,
  author  = {Edelmann, Dominik},
  title   = {Finite Element Analysis for a Diffusion Equation on a Harmonically Evolving Domain},
  journal = {IMA Journal of Numerical Analysis},
  volume  = {42},
  number  = {2},
  pages   = {1866--1901},
  year    = {2022},
  doi     = {10.1093/imanum/drab026}
}

@article{ElliottRanner2021,
  author  = {Elliott, Charles M. and Ranner, Thomas},
  title   = {A Unified Theory for Continuous-in-time Evolving Finite Element Space Approximations to Partial Differential Equations in Evolving Domains},
  journal = {IMA Journal of Numerical Analysis},
  volume  = {41},
  number  = {3},
  pages   = {1696--1845},
  year    = {2021},
  doi     = {10.1093/imanum/draa062}
}

@article{LiXiaYang2023,
  author  = {Li, Buyang and Xia, Yinhua and Yang, Zongze},
  title   = {Optimal Convergence of Arbitrary {Lagrangian--Eulerian} Iso-parametric Finite Element Methods for Parabolic Equations in an Evolving Domain},
  journal = {IMA Journal of Numerical Analysis},
  volume  = {43},
  number  = {1},
  pages   = {501--534},
  year    = {2023},
  doi     = {10.1093/imanum/drab099}
}

@article{GanesanTobiska2012,
  author  = {Ganesan, Sashikumaar and Tobiska, Lutz},
  title   = {Arbitrary {Lagrangian--Eulerian} Finite-element Method for Computation of Two-phase Flows with Soluble Surfactants},
  journal = {Journal of Computational Physics},
  volume  = {231},
  number  = {9},
  pages   = {3685--3702},
  year    = {2012},
  doi     = {10.1016/j.jcp.2012.01.018}
}

@article{YangDong2019,
  author  = {Yang, Zhiguo and Dong, Suchuan},
  title   = {An Unconditionally Energy-stable Scheme Based on an Implicit Auxiliary Energy Variable for Incompressible Two-phase Flows with Different Densities Involving Only Precomputable Coefficient Matrices},
  journal = {Journal of Computational Physics},
  volume  = {393},
  pages   = {229--257},
  year    = {2019},
  doi     = {10.1016/j.jcp.2019.05.018}
}

@article{ZhaoRen2020,
  author  = {Zhao, Quan and Ren, Weiqing},
  title   = {An Energy-stable Finite Element Method for the Simulation of Moving Contact Lines in Two-phase Flows},
  journal = {Journal of Computational Physics},
  volume  = {417},
  pages   = {109582},
  year    = {2020},
  doi     = {10.1016/j.jcp.2020.109582},
  eprint  = {2002.12009},
  archiveprefix = {arXiv}
}

@misc{MaRao2026,
  author        = {Ma, Shu and Rao, Qiqi},
  title         = {High-order Energy-stable {BGN} Parametric Finite Element Methods for Geometric Flows},
  year          = {2026},
  note          = {arXiv:2608.29877},
  eprint        = {2608.29877},
  archiveprefix = {arXiv},
  primaryclass  = {math.NA},
  doi           = {10.48550/arXiv.2608.29877}
}

\end{document}